\documentclass[reqno]{amsart}

\usepackage{graphicx}
\usepackage{tikz}
\usepackage[subpreambles=true]{standalone}
\usepackage{amsmath}
\usepackage{amssymb,mathrsfs}
\usepackage{amsthm,mathtools}

\usepackage[colorlinks=true, pdfstartview=FitV, linkcolor=blue, citecolor=blue, urlcolor=blue]{hyperref}
\usepackage[nameinlink]{cleveref}

\newcommand{\R}{\mathbb R}

\newcommand{\K}{\mathbb K}
\DeclareMathOperator{\gin}{gin}
\DeclareMathOperator{\inid}{in_<}
\DeclareMathOperator{\reg}{reg}
\DeclareMathOperator{\areg}{\widehat{\operatorname{reg}}}
\DeclareMathOperator{\SP}{SP}
\DeclareMathOperator{\NP}{NP}
\DeclareMathOperator{\wald}{\widehat \alpha}

\let\emptyset \varnothing

\theoremstyle{plain}
\newtheorem{thm}{Theorem}[section]
\newtheorem*{thm*}{Theorem}
\newtheorem{lem}[thm]{Lemma}
\newtheorem{prop}[thm]{Proposition}
\newtheorem{cor}[thm]{Corollary}
\newtheorem{conj}[thm]{Conjecture}

\newtheorem{introthm}{Theorem}

\theoremstyle{definition}
\newtheorem{notation}[thm]{Notation}
\newtheorem{example}[thm]{Example}

\newtheorem{obs}[thm]{Observation}
\newtheorem{defn}[thm]{Definition}
\AtBeginEnvironment{defn}{}

\numberwithin{equation}{section}

\title{Asymptotic~invariants of~symbolic~powers
	of~binomial~edge~ideals}

\author[D. Belotserkovskiy]{Dennis Belotserkovskiy}
\address{Dennis Belotserkovskiy, Department of Mathematics, Colgate University, 13 Oak Drive, Hamilton NY, 13346, USA}
\email{dbelotserkovskiy@colgate.edu}
\author[M. Landín]{Mariana Landín}
\address{Mariana Landín, Department of Mathematics, Universidad de Colima, C. Bernal Díaz del Castillo 340, Villas San Sebastián, 28045 Colima, Colima, México.}
\email{mlandin1@ucol.mx}
\author[C. Ruppe]{Charlie Ruppe}
\address{Charlie Ruppe, Department of Mathematics and Statistics, Carleton College, 1 North College Street, Northfield, MN 55057, USA}
\email{ruppec@carleton.edu}
\author[L. Teryoshin]{Lizzy Teryoshin}
\address{Lizzy Teryoshin, 
	Department of Mathematics, University of California, San Diego, 9500 Gilman Drive \#0112, La Jolla, CA 92093, USA}
\email{eteryoshin@ucsd.edu}


\begin{document}
	\begin{abstract}
		To a graph $G$ one associates the binomial edge ideal $J_G$  generated by a collection of binomials corresponding to the edges of $G$. In this paper, we study the asymptotic behavior of symbolic powers of $J_G$, its lexicographic initial ideal $\inid(J_G)$, and its multigraded generic initial ideal $\gin(J_G)$. We focus on the Waldschmidt constant, $\widehat{\alpha}$, and asymptotic regularity, $\widehat{\reg}$, which capture linear growth of minimal generator degrees and Castelnuovo--Mumford regularity. 
		We explicitly compute $\wald(J_G)$ and $\wald(\inid(J_G))$, and compare the Betti numbers of the symbolic powers of $J_G$ and $J_H$, where $H$ is a subgraph of $G$. 
		To analyze $\inid(J_G)$ and $\gin(J_G)$, we use the symbolic polyhedron, a convex polyhedron that encodes the elements of the symbolic powers of a monomial ideal. We determine its vertices via $G$'s induced connected subgraphs and show that $\wald(\gin(J_G))=\wald(I_G)$, where $I_G$ is the edge ideal of $G$. This yields an alternate proof of known bounds for $\wald(I_G)$ in terms of $G$'s clique number and chromatic number.
	\end{abstract}
	
	\maketitle
	\tableofcontents
	
	\section{Introduction}
	In recent years, there has been extensive research on the properties of symbolic powers of ideals,  
	mostly through comparison with their ordinary powers \cite{bocci2010,ein_uniform_2001,hochster2002,ma18}. In this paper, we consider symbolic powers of binomial edge ideals, their initial ideals, and multigraded generic initial ideals, in polynomial rings over fields. In particular, we study their asymptotic behavior through two invariants: the Waldschmidt constant and the asymptotic regularity, which capture linear growth of minimal generator degrees and of the Castelnuovo–Mumford regularity, respectively.
	
	Binomial edge ideals were first introduced in \cite{herzog2010binomial}
	and independently in \cite{ohtani_graphs_2011} as a means of associating
	algebraic structures to graphs. They are generated by binomials corresponding to the
	edges of a simple graph. Nowadays, this class of ideals represents an active area of research in combinatorial commutative algebra, with
	applications to algebraic statistics as ideals generated by conditional independence statements \cite{herzog2010binomial}.
	
	The Waldschmidt constant, appearing implicitly in \cite{Waldschmidt1977b} and first defined explicitly in \cite{bocci2010}, captures the asymptotic growth of the initial degree of the symbolic powers of an ideal. 
	If $I$ is a non-zero homogeneous ideal, its initial degree is defined as
	$\alpha(I) = \min\{\deg(f)\mid f\in I\setminus\{0\}\}$. The Waldschmidt
	constant of $I$ is defined as the rate of growth of the initial degree for the sequence of symbolic powers $(I^{(m)})_{m\ge1}$, specifically 
	$$\widehat{\alpha}(I)=\lim_{m\to\infty}\frac{\alpha(I^{(m)})}{m}.
	$$
	This limit exists for any homogeneous ideal \cite{bocci2010} . The
	asymptotic growth factor of an arbitrary valuation applied to the symbolic
	powers of an ideal is dubbed the skew Waldschmidt constant in
	\cite{DiPasquale19}.

	Another invariant of interest is the Castelnuovo-Mumford regularity, denoted $\reg()$. For an introduction to this topic and an account of its significance, see \cite[Chapter 20.5]{Eisenbud}.
	The asymptotic behavior of the Castelnuovo-Mumford regularity
	has been studied for various families of ideals, including ordinary powers, symbolic powers, and integral closures of powers \cite{CHT,ha_asymptotic_2025,Jayanthan_2020,vijay2000}. For ordinary powers, it is known that $\reg(I^m)$ is asymptotically linear in $m$ \cite{CHT,vijay2000}, and bounds on $\reg(I^m)$ have been computed explicitly for various families of ideals including certain binomial edge ideals, ideals generated by quadratic sequences, and edge ideals \cite{Jayanthan_2020,banerjee2014}.
	For symbolic powers $(I^{(m)})_{m\ge1}$, the existence of linear bounds on the regularity $\reg(I^{(m)})$
	has been thoroughly investigated \cite{ein_saturation_2022,ha_asymptotic_2025,herzog_asymptotic_2002}. However, $\reg(I^{(m)})$ is not always linear for large $m$, even for squarefree monomials $I$. For example, the regularity of symbolic powers of cover ideals of corona graphs is known to not be eventually linear \cite{dung21}.
	The asymptotic regularity of a homogeneous ideal $I$ is the value of the limit $$
	\widehat{\reg}(I)=\lim_{m\to\infty}\frac{\reg(I^{(m)})}{m},
	$$provided that the limit exists. There is evidence that the asymptotic regularity may not exist for some graded ideals \cite{ha_newton-okounkov_2021}.
	
	Both the Waldschmidt constant and asymptotic regularity are known to exist for monomial ideals \cite{dung21}. They can be computed for monomial ideals using a recurrent tool in this paper, the symbolic polyhedron, which is a convex body first introduced in \cite{CooperSP} and thoroughly investigated in \cite{SP,ConvexBodies, ha_newton-okounkov_2021}.
	For a monomial ideal $I$, the symbolic polyhedron $SP(I)$ is the union of the convex hulls of lattice points corresponding to the exponent vectors of monomials in $I^{(q)}$ scaled by $1/q$.
	It is shown in \cite{ConvexBodies} that the minimal coordinate sum of the vertices of $SP(I)$ is $\widehat\alpha(I)$, and in \cite{dung21} that the maximal coordinate sum of the vertices of $SP(I)$ is $\widehat{\operatorname{reg}}(I)$.
	
	The structure of the paper is as follows. In section 2, we present definitions and results needed throughout the paper. In section 3, we discuss the Waldschmidt constant and asymptotic regularity of binomial edge ideals and their initial ideals. Using the prime decompositions of these ideals, we are able to compute the Waldschmidt constant for any binomial edge ideal or initial ideal. We are also able to compute the asymptotic regularity in the case that the ordinary and symbolic powers of the binomial edge ideal coincide. This is the case, for example, for the class of closed graphs (\Cref{def: closed graph and labeling}). 
	In particular, we show: 
	\begin{introthm}[\Cref{thm: waldschmidt J_G}, \Cref{closed_graph_cor}]
		\label{intro:waldschmidt J_G}
		For all graphs $G$, $\widehat{\alpha}(J_G) = \widehat{\alpha}(\mathrm{in}(J_G)) = 2$. 
		For  closed graphs $G$ with lexicographic ordering $<$, $\widehat{\reg}(J_G)=\widehat{\reg}(\inid(J_G))=2$.
	\end{introthm}
	
	We also compare the Betti numbers of symbolic powers of the binomial edge ideal of a graph to those of induced subgraphs. In particular,  
	if $H$ is an induced subgraph of $G$ we establish in \Cref{prop: betti subgraph}  the inequalities  $\beta_{i,j}(S/J_H^{(m)})\le\beta_{i,j}(S/J_G^{(m)})$ and $\widehat{\reg}(J_H)\le\widehat{\reg}(J_G)$.
	
	In section 4, we discuss symbolic polyhedra, particularly in relation to (multigraded generic) initial ideals of binomial edge ideals. We are able to show that the Waldschmidt constant and asymptotic regularity can be found by taking a minimum and maximum, respectively, over sets of full vertices (see \Cref{def:full_vertices}) associated to induced subgraphs of $G$:
	\begin{introthm}[\Cref{cor:subgraph_decomp}]
		\label{intro:subgraph_decomp} 
		Let $G$ be a connected
		graph on $n$ vertices, and let $\mathcal{H}_G$ denote the set of connected induced subgraphs of
		$G$ with at least one edge. With $I_H=\mathrm{gin}(J_H)$ or $I_H=\inid (J_H)$
		for $H\in \mathcal{H}_G$, under canonical projections $i_H:\mathbb{R}^{2|H|}\rightarrow \mathbb{R}^{2n}$, the set $\mathcal{V}(\SP(I_G))$ of vertices of $\SP(I_G)$ can be recovered from the sets of full vertices $\mathcal{V}_F(\SP(I_H))$ of graphs in $\mathcal{H}_G$ as follows
		\begin{equation}\label{eq: full} \mathcal{V}(\SP(I_G))=\bigcup_{H\in
				\mathcal{H}_G} i_H(\mathcal{V}_F(\SP(I_H))).
		\end{equation}
		Furthermore, 
		\begin{align*}
			\wald(I_G) &= \text{the smallest sum of coordinates of an element in \eqref{eq: full}}; \\
			\widehat{\operatorname{reg}}(I_G) &= \text{the largest sum of coordinates of an element in \eqref{eq: full}}.
		\end{align*}
	\end{introthm} 
	Finally, in section 5 we use the results developed in section 4 to calculate the Waldschmidt constant and asymptotic regularity of generic initial ideals. In particular, we are able to directly relate the Waldschmidt constant of the generic initial ideal to the chromatic number and clique number of the graph as well as to the Waldschmidt constant of the edge ideal of $G$:
	\begin{introthm}[\Cref{thm:chi_omega_bound}]
		\label{intro:chi_omega_bound}
		Let $G$ be a non-empty graph with chromatic number $\chi(G)$ and clique number $\omega(G)$. Then
		$$\frac{\chi(G)}{\chi(G)-1}\leq \widehat\alpha(\gin(J_G))=\wald(I_G)\leq \frac{\omega(G)}{\omega(G)-1}.$$
	\end{introthm}
	This then allows us to compute the Waldschmidt constant of the generic initial ideals of certain classes of graphs, including weakly perfect and bipartite graphs.

\section{Background}
In this section, we summarize notions related to binomial edge ideals,  symbolic powers, and their asymptotic values. We also define symbolic polyhedra and describe their useful applications. Throughout $\K$ denotes a field.

\subsection{Binomial edge ideals}
\begin{defn}
	Given a graph $G=(V,E)$ with vertex set $V(G)=\{1,2,\ldots,n\}=[n]$ and edge set $E(G)$, the \textbf{binomial edge ideal} $J_G \subset S = \mathbb K[x_1,\ldots,x_n,y_1,\ldots,y_n]$ is defined as
	$$ J_G = (x_iy_j - x_jy_i \mid \{i,j\} \in E(G)).$$
\end{defn}
In this paper $G$ is always assumed to be a simple graph.
\begin{defn}
	\label{init}
	The \textbf{leading monomial} of a polynomial $ f \ne 0 $ with respect to a given term order $ < $ is the largest monomial appearing in $ f $, denoted $ LM_<(f) $.
	The \textbf{initial ideal} of $J_G \subset S$ with respect to a given term order $ < $ is the ideal
	$$
	\inid(J_G) = \big(LM_<(f) \mid f\in J_G\big).
	$$
\end{defn}

In this paper, the initial ideal $\inid(J_G)$ is always taken with respect to the lexicographic order on $S$ induced by $x_1>\dots>x_n>y_1>\dots>y_n$. The generators of $\inid(J_G)$ can be expressed in terms of admissible paths, which in \cite{herzog2010binomial} are defined as follows.

\begin{defn}[{\cite[p.~6]{herzog2010binomial}}]
	Let $G=(V,E)$ be a simple graph on $[n]$, and take $i,j\in [n]$, $i<j$. A path $\pi=\{i=i_0, i_1,\dots, i_r= j\}$ is called an {\bf admissible path} if the following are satisfied:
	\begin{itemize}
		\item For $k\in \{1,\ldots,r-1\}$, $(i_k,i_{k+1})\in E$.
		\item For $k\in \{1,\ldots,r-1\}$, either $i_k <i$ or $i_k >j$.
		\item For any subset $\{j_1,\ldots,j_s\}\subsetneq\{i_1,\ldots ,i_{r-1}\}$, $\{i,j_1,\ldots,j_s,j\}$ is not a path.
	\end{itemize}
	If $\pi$ is an admissible path, let $u_\pi=\left(\prod_{i_k<i} x_{i_k}\right)\left(\prod_{i_\ell>j}y_{i_\ell}\right)$.
\end{defn}
\begin{thm}\cite[Theorem 2.1]{herzog2010binomial}
	\label{thm:inid_gen}
	For a graph $G$, using the above notation,
	$$\inid(J_G)=(LM_<(u_\pi(x_iy_j-x_jy_i))\mid \pi \text{ an admissible path from } i \text{ to } j).$$
\end{thm}

Another relevant monomial ideal is the \textbf{multigraded generic initial ideal}, denoted $\gin(I)$. We define a multigrading on $S$  by setting $ \deg(x_i) = \deg(y_i) = \mathbf{e}_{i} \in \mathbb{N}^n$, where $ \mathbf{e}_{i}$  has 1 in the $ i $-th position and 0 elsewhere.

\begin{thm}\cite[Theorem 5.1]{conca2021radicalgenericinitialideals}
	\label{thm:gin_gen}
	The multigraded generic initial ideal of a binomial edge ideal $ J_G $ is
	$$
	\mathrm{gin}(J_G) = (y_{v_1} \cdots y_{v_s} x_i x_j \mid i, v_1, \dots, v_s, j \text{ a path in } G).
	$$
\end{thm}
Furthermore, $y_{v_1} \cdots y_{v_s} x_i x_j$ is a minimal generator of $\mathrm{gin}(J_G)$ if the subgraph of $G$ induced by $\{i,v_1,\dots,v_s,j\}$ is isomorphic to the path graph. We refer to a subgraph isomorphic to the path graph as an induced path.

We next consider prime decompositions for the initial ideal and the multigraded generic initial ideal.
To do so, we define a class of relevant prime ideals.

\begin{notation}
	\label{not:P_T}
	Let $G$ be a graph on $[n]$, and let $T \subseteq [n]$. Let $G_1,\ldots G_{c_G(T)}$ be the connected components of the induced subgraph of $G$ on vertices $[n]\setminus T$. For any $1\leq i \leq c_{G}(T)$, we denote by $\widetilde{G_i}$ the complete graph on the vertex set $V(G_i)$. Let
	\begin{equation}\label{eq:P_T}
		P_T(G) := (x_i,y_i \mid i\in T)+\sum_{i=1}^{c_G(T)}J_{\widetilde{G_i}}.
	\end{equation}
	For any $T$, the ideal $P_T(G)$ is a prime ideal in $S$.
\end{notation}
\begin{defn}
	\label{defn:disconnecting_sets}
	As in \cite{herzog2010binomial}, we define $T$ to be an \textbf{irredundant disconnecting set} (\textbf{IDS}) for a graph $G$ if for all $i\in T$, $c_G(T)>c_G(T\setminus\{i\})$. Note that $\emptyset$ is vacuously an IDS.
	We denote by $\mathcal{T}(G)$ the set of all irredundant disconnecting sets of $G$:
	$$
	\mathcal{T}(G) := \{T \subseteq [n] \mid T \text{ an IDS}\}.
	$$
\end{defn}

\begin{thm}[{\cite[Corollary 3.9]{herzog2010binomial}}, {\cite[Theorem 2.1, Corollary 1.12]{conca}}]
	\label{thm:beiprimedecomp}
	For any graph $G$, the irredundant prime decomposition of $J_G$ is
	$$
	J_G=\bigcap_{T\in \mathcal{T}(G)} P_T(G).
	$$
	Moreover, for any term ordering $<$, the initial ideal $\inid(J_G)$ is radical and has decomposition
	$$
	\inid(J_G)=\bigcap_{T\in \mathcal{T}(G)} \inid(P_T(G)),
	$$
\end{thm}

While $\inid (P_T(G))$ is not generally prime, this decomposition allows us to find minimal primes of $\inid(J_G)$ and $\gin(J_G)$. 
\begin{prop}[{\cite[Lemma 2.1]{lerda}}]
	\label{prop:in_prime_decomp}
	The minimal primes of $\inid(J_G)$ are exactly the ideals
	$$Q_{T,U} = (x_i,y_i \mid i\in T)  + \sum_{i=1}^{c_G(T)} (x_i \mid i\in G_i, \, i < u _i) +  (y_i \mid i\in G_i, \, i > u _i)$$
	where $T$ is an IDS of $G$ and $U=\{u_1,\dots,u_{c_G(T)}\}$, where $u_i$ is some vertex in the connected component $G_i$ of $G \setminus T$.
\end{prop}

\begin{prop}[{\cite[Proposition 2.2]{conca}}]
	\label{prop:new_gin_prime_decomp}
	The minimal primes of $\gin(J_G)$ are exactly the ideals
	$$P_{T,U} = (x_i,y_i \mid i\in T)  + \sum_{i=1}^{c_G(T)} (x_i \mid i\in G_i, \, i \neq u _i)$$
	where $T$ is an IDS of $G$ and $U=\{u_1,\dots,u_{c_G(T)}\}$, where $u_i$ is some vertex in the $i$-th connected component of $G \setminus T$.
\end{prop}
Note that $\inid(P_T)=\bigcap_U Q_{T,U}$ and $\gin(P_T)=\bigcap_U P_{T,U}$.

\subsection{Symbolic powers and asymptotic invariants}

A classical result in commutative algebra is that every ideal in a Noetherian ring has a primary decomposition. In particular, radical ideals have a unique minimal prime decomposition which is used to define symbolic powers of an ideal.
\begin{defn}
	If $P$ is a prime ideal, and $m \ge 1$ is an integer, the \textbf{symbolic power} $P^{(m)}$ is the ideal
	$$
	P^{(m)} = \{ r \in R \mid rq \in P^m \text{ for some } q \notin P \}.
	$$
	
	If $I$ is a radical ideal with minimal prime decomposition
	$$
	I = P_1 \cap P_2 \cap \cdots \cap P_s,
	$$
	then the \textbf{symbolic power} $I^{(m)}$ is the ideal
	$$
	I^{(m)} = P_1^{(m)} \cap P_2^{(m)} \cap \cdots \cap P_s^{(m)}.
	$$
\end{defn}

We aim to characterize the asymptotic behavior of symbolic powers of binomial edge ideals and their relevant monomial ideals. There are two particular invariants we wish to compute.
\begin{defn}
The {\bf initial degree} of a homogeneous ideal $I$ is 
$$\alpha(I)=\min\{\deg(f)\mid 0\neq f\in I\}.$$
The {\bf Waldschmidt constant} of $I$ is $$\widehat{\alpha}(I)=\lim\limits_{m\to \infty}\frac{\alpha(I^{(m)})}{m}=\inf\limits_{m\geq 1}\frac{\alpha(I^{(m)})}{m}.$$
\end{defn}
The Waldschmidt constant describes the long-term ``starting degree'' of the ideal's symbolic powers, and \cite[Lemma 2.3.1]{bocci2010} shows that it exists for any homogeneous ideal $I$. Another invariant similarly captures the asymptotic behavior of the Castelnuovo-Mumford regularity: 
\begin{defn}
The {\bf asymptotic regularity} of a homogeneous ideal $I$ is 
$$
\widehat{\reg}(I)=\lim\limits_{m\to \infty}\frac{\reg(I^{(m)})}{m}.
$$
\end{defn}

The asymptotic regularity is known to exist for monomial ideals \cite{dung21}. Both $\wald(I)$ and $\widehat{\reg}(I)$ can be computed via a convex body that encodes the monomials contained in an ideal and its symbolic powers.

\subsection{Symbolic polyhedra}
We arrive at the most important tool used in this paper. The symbolic polyhedron has equivalent definitions that we use as necessary.
\begin{defn}
\label{defn:newton_polyhedron}
Let $I\subset \mathbb{K}[x_1,\ldots,x_n]$ be a monomial ideal. The
\emph{Newton polyhedron} of $I$, denoted $\NP(I)$, is defined by
$$\NP(I)=\text{conv}\{(\alpha_1,\dots,\alpha_n)\mid x_1^{\alpha_1}\cdots
x_n^{\alpha_k}\in I\}\subset \mathbb{R}^n_{\geq 0},$$ where $\text{conv}$ is the convex hull.

The \emph{symbolic polyhedron} of $I$, denoted $\SP(I)$, is defined as \begin{align*} \SP(I)&=\bigcup _{q\geq 1}
	\frac{NP(I^{(q)})}{q}=
	\text{conv}\left\{\frac aq \; \middle|\; x^a \in I^{(q)} \text{ for some }
	q\geq 1, a \in \mathbb{Z}_{\geq 0}^n\right\}.
\end{align*}
\end{defn}
Symbolic polyhedra were introduced by Cooper-Embree-H\'a-Hoefel in \cite{CooperSP}. They fit within a more general family of Newton-Okounkov convex bodies which can be defined for graded families of ideals \cite{ha_newton-okounkov_2021}. \Cref{defn:newton_polyhedron} is based on the latter perspective, also called a limiting body in \cite{ConvexBodies}. 

Symbolic polyhedra were originally defined as intersections of Newton polyhedra. As a special case of this alternate definition, if $I$ is squarefree, $\SP(I)$ can be defined per {\cite[Remark 3.3]{SP}} as follows: 
\begin{defn}
	\label{defn:facets}
	For a squarefree monomial ideal $I\subset \mathbb{K}[x_1,\ldots,x_n]$ with minimal primes $P_M=(x_j\mid j\in M)$, $M\subset[n]$, $\SP(I)$ is equivalently defined as
	\begin{align*}
		\SP(I)&= \left\{a\in\mathbb{R}_{\geq 0}^n \;\middle|\;
		\sum_{j \in M} a_j \geq 1 \text{ for all } P_M\right\}
	\end{align*}
	Each $P_M$ corresponds to a
	\emph{one-facet} of $\SP(I)$ denoted $F_M$, whose defining inequality is $\sum_{j\in M} x_j\geq 1$.
	
	There are also $n$ \emph{zero-facets} of $\SP(I)$, one for each $x_i$
	defined by $x_i\geq 0$. The \emph{facets} of $\SP(I)$ are the collection of
	its one-facets and zero-facets. For a point $a=(a_1,\dots,a_n)\in
	\mathbb{R}^n$, $a\in \SP(I)$ if and only if the defining inequalities of
	every facet of $I$ are satisfied by $a$. A point
	$a\in \mathbb{R}^n$ is a \emph{vertex} of $\SP(I)$ if and only if for any non-zero
	vector $\textbf{u}\in \mathbb{R}^n$ and $\epsilon>0$, $a+\epsilon\textbf{u}\in
	\SP(I)$ implies $a-\epsilon\textbf{u}\notin \SP(I)$.
\end{defn}

\Cref{fig:symbolic-polyhedra} shows examples of symbolic polyhedra with $n=3$.

\begin{figure}[h!]
	\centering
		
	\begin{tikzpicture}[]
		\tikzstyle{every node}=[font=\scriptsize]
		
		\pgfmathsetmacro{\extent}{2}
		\draw[dashed, ->] (0,0,0) -- (${\extent + 2}*(1,0,0)$) node[anchor=north east]{x};
		\draw[dashed, ->] (0,0,0) -- (${\extent + 2}*(0,1,0)$) node[anchor=north east]{y};
		\draw[dashed, ->] (0,0,0) -- (${\extent + 2}*(0,0,1)$) node[anchor=north east]{z};
		\coordinate (A) at (1,1,0);
		\coordinate (B) at (1,0,1);
		\coordinate (C) at (0,1,1);
		\coordinate (D) at (0.5,0.5,0.5);
		
		\coordinate (A1) at ($(A) + \extent*(1,0,0)$);
		\coordinate (A2) at ($(A) + \extent*(0,1,0)$);
		\coordinate (B1) at ($(B) + \extent*(1,0,0)$);
		\coordinate (B3) at ($(B) + \extent*(0,0,1)$);
		\coordinate (C2) at ($(C) + \extent*(0,1,0)$);
		\coordinate (C3) at ($(C) + \extent*(0,0,1)$);
		\filldraw[fill=blue!50, opacity=0.5] (A) -- (B) -- (C) -- cycle;
		\filldraw[fill=blue!50, opacity=0.5] (A) -- (B) -- (D) -- cycle;
		\filldraw[fill=blue!50, opacity=0.5] (B) -- (C) -- (D) -- cycle;
		\filldraw[fill=blue!50, opacity=0.5] (C) -- (A) -- (D) -- cycle;
		\filldraw[fill=blue!30, opacity=0.5] (A) -- (A1) -- (A2) -- cycle;
		\filldraw[fill=blue!45, opacity=0.5] (A) -- (A2) -- (C2) -- (C) -- cycle;
		\filldraw[fill=blue!30, opacity=0.5] (B) -- (B1) -- (B3) -- cycle;
		\filldraw[fill=blue!45, opacity=0.5] (B) -- (B3) -- (C3) -- (C) -- cycle;
		\filldraw[fill=blue!30, opacity=0.5] (C) -- (C2) -- (C3) -- cycle;
		\filldraw[fill=blue!45, opacity=0.5] (A) -- (A1) -- (B1) -- (B) -- cycle;
		
		\draw[thick, ->] (A) -- (A1);
		\draw[thick, ->] (A) -- (A2);
		\draw[thick, ->] (B) -- (B1);
		\draw[thick, ->] (B) -- (B3);
		\draw[thick, ->] (C) -- (C2);
		\draw[thick, ->] (C) -- (C3);
	\end{tikzpicture}
	
	\begin{tikzpicture}[]
		\tikzstyle{every node}=[font=\scriptsize]
		
		\pgfmathsetmacro{\extent}{2}
		\draw[dashed, ->] (0,0,0) -- (${\extent + 2}*(1,0,0)$) node[anchor=north east]{x};
		\draw[dashed, ->] (0,0,0) -- (${\extent + 2}*(0,1,0)$) node[anchor=north east]{y};
		\draw[dashed, ->] (0,0,0) -- (${\extent + 2}*(0,0,1)$) node[anchor=north east]{z};
		\coordinate (A) at (1,0,0);
		\coordinate (B) at (0,1,0);
		\coordinate (C) at (0,0,1);
		
		\coordinate (A1) at ($(A) + \extent*(1,0,0)$);
		\coordinate (B2) at ($(B) + \extent*(0,1,0)$);
		\coordinate (C3) at ($(C) + \extent*(0,0,1)$);
		\filldraw[fill=blue!45, opacity=0.5] (A) -- (B) -- (C) -- cycle;
		\filldraw[fill=blue!30, opacity=0.5] (A) -- (A1) -- (B2) -- (B) -- cycle;
		\filldraw[fill=blue!30, opacity=0.5] (B) -- (B2) -- (C3) -- (C) -- cycle;
		\filldraw[fill=blue!30, opacity=0.5] (C) -- (C3) -- (A1) -- (A) -- cycle;
		
		\draw[thick, ->] (A) -- (A1);
		\draw[thick, ->] (B) -- (B2);
		\draw[thick, ->] (C) -- (C3);
	\end{tikzpicture}
	
	\begin{tikzpicture}[]
		\tikzstyle{every node}=[font=\scriptsize]
		
		\pgfmathsetmacro{\extent}{2}
		\draw[dashed, ->] (0,0,0) -- (${\extent + 2}*(1,0,0)$) node[anchor=north east]{x};
		\draw[dashed, ->] (0,0,0) -- (${\extent + 2}*(0,1,0)$) node[anchor=north east]{y};
		\draw[dashed, ->] (0,0,0) -- (${\extent + 2}*(0,0,1)$) node[anchor=north east]{z};
		\coordinate (A) at (0,1,0);
		\coordinate (B) at (1,0,1);
		
		\coordinate (A1) at ($(A) + \extent*(1,0,0)$);
		\coordinate (A2) at ($(A) + \extent*(0,1,0)$);
		\coordinate (A3) at ($(A) + \extent*(0,0,1)$);
		\coordinate (B1) at ($(B) + \extent*(1,0,0)$);
		\coordinate (B2) at ($(B) + \extent*(0,1,0)$);
		\coordinate (B3) at ($(B) + \extent*(0,0,1)$);
		\filldraw[fill=blue!45, opacity=0.5] (A) -- (A1) -- (B1) -- (B) -- cycle;
		\filldraw[fill=blue!45, opacity=0.5] (A) -- (A3) -- (B3) -- (B) -- cycle;
		
		\filldraw[fill=blue!30, opacity=0.5] (A) -- (A1) -- (A2) -- cycle;
		\filldraw[fill=blue!30, opacity=0.5] (A) -- (A3) -- (A2) -- cycle;
		
		\filldraw[fill=blue!30, opacity=0.5] (B) -- (B1) -- (B3) -- cycle;
		
		\draw[thick, ->] (A) -- (A1);
		\draw[thick, ->] (A) -- (A2);
		\draw[thick, ->] (A) -- (A3);
		\draw[thick, ->] (B) -- (B1);
		\draw[thick, ->] (B) -- (B3);
	\end{tikzpicture}
	\caption{Symbolic polyhedra corresponding to the ideals $(xy,xz,yz)$, $(x,y,z)$, and $(xz,y)$ respectively.}
	\label{fig:symbolic-polyhedra}
\end{figure}
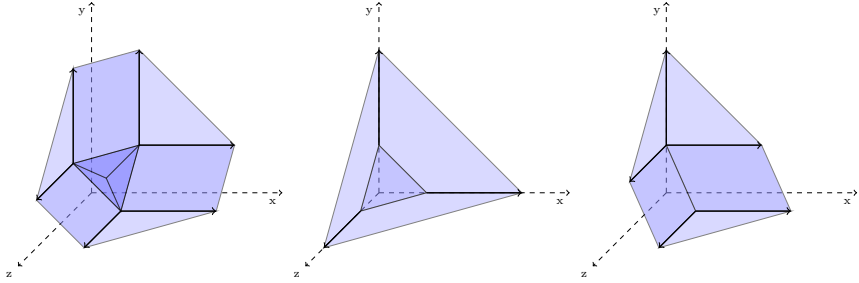

Keep in mind the distinction between the vertices of a symbolic polyhedron and the vertices of a graph. In a geometric context, facets and vertices are faces of the symbolic polyhedron.

\begin{defn}[{\cite[section 1.5]{kaibel2010basicpolyhedraltheory}}] \label{defn:faces}
	Let $\mathcal{P}$ be a polyhedron in $\mathbb{R}^n$ and $H^\leq$ a
	halfspace containing $\mathcal{P}$ defined by the inequality $a\cdot x\leq b$, $a\in
	\mathbb{R}^n$ and $b\in \mathbb{R}$.
	
	A \emph{face} of $\mathcal{P}$ is defined by $\mathcal{F}=\mathcal{P}\cap H^=$, where
	$H^==\{x\mid a\cdot x=b\}$ is a plane in $\mathbb{R}^n$. If $\mathcal{F}\notin
	\{\varnothing,\mathcal{P}\}$, $\mathcal{F}$ is called a \emph{proper face}. In alignment with
	\Cref{defn:facets}, \emph{facets} are $n-1$ dimensional faces of $\mathcal{P}$, and
	\textbf{vertices} are $0$ dimensional faces of $\mathcal{P}$. Note
	that faces are themselves convex polyhedra. The vertices of a face $\mathcal{F}$ are vertices of $\mathcal{P}$ that lie on $\mathcal{F}$ \cite[section
	2.1]{grunbaum_convex_1969}. We denote the set of vertices of $\SP(I)$ as $\mathcal{V}(\SP(I))$.
\end{defn}

The symbolic polyhedron determines the asymptotic invariants of a squarefree monomial ideal as follows:
\begin{thm}[{\cite[Corollary 4.6]{ConvexBodies}, \cite[Theorem 1.3]{dung21}}]
	\label{thm:wald&areg_via_SP}
	For a squarefree monomial ideal $I\in \mathbb{K}[x_1,\dots,x_n]$, we have
	\begin{align*}
		\wald(I) = \min&\left\{ \sum_{i=1}^n a_i \;\middle|\;  a\in \mathcal{V}(\SP(I)) \right\},\\
		{\widehat \reg}(I) = \max&\left\{ \sum_{i=1}^n a_i \;\middle|\;  a\in \mathcal{V}(\SP(I)) \right\}.
	\end{align*}
\end{thm}

\section{Results for binomial edge ideals and the initial ideal}

We first state some useful lemmas about the associated primes of our relevant ideals. These lemmas will be used to compare invariants of a graph to those of its subgraphs throughout the paper.
\begin{lem}
	\label{lem:primedecompositioncontainment}
	Let $I\subseteq J$ be radical ideals. 
	If $J\subset P_J$, $P_J$ a prime, then there exists a minimal prime $P_I$ of $I$ such that $I\subseteq P_I\subseteq P_J$.
\end{lem}
\begin{proof}
	Given an irredundant prime decomposition $I=\bigcap P_I\subset P_J$, assume towards a contradiction that $P_I\not\subseteq P_J$ for all minimal primes $P_I$ of $I$. For each $P_I$, pick $f_{P_I}\in P_I\backslash P_J$. Since $P_J$ is prime, $ \prod_{P_I} f_{P_I}\notin P_J$, but $\prod_{P_I} f_{P_I}\in \bigcap P_I\subset P_J$. This is a contradiction, so the result holds.
\end{proof}

\begin{lem}
	\label{lem:primedecompcontainment}
	For a graph $G$, let $I_G$ be either the binomial edge ideal $J_G$, $\gin (J_G)$, or $\inid (J_G)$. Let $H$ be an induced subgraph of $G$, $I_H$ the ideal among $J_H$, $\gin (J_H)$, or $\inid (J_H)$ that matches the choice for $I_G$, and let $P_H$ be a minimal prime of $I_H$. There exists a minimal prime $P_G$ of $I_G$ such that $P_H \subset P_G$.
\end{lem}
\begin{proof}
	Let $V_G$ denote the set of vertices of $G$ and use similar notation for its subgraphs.
	
	\textbf{Case 1: $I_G:=J_G$}
	
	By \Cref{thm:beiprimedecomp}, there is an IDS $T$ of $H$ such that
	$$ P_H = (x_i,y_i \mid i \in T) + J_{\widetilde{H_1}} + \cdots + J_{\widetilde{H_{c_H(T)}}},$$
	where the $H_i$ are the connected components of $H\setminus T$.
	
	First, we construct an IDS that defines a suitable minimal prime of $J_G$.
	Consider $T_0 = T \cup (V_G\setminus V_H) \subset V_G$, with $c_G(T_0)=c_H(T)$. If there is a vertex $v\in T_0\setminus T$ such that $c_G(T_0\setminus \{v\})\geq c_G(T_0)$, set $T_1=T_0\setminus \{v\}$. Repeating this inductively we arrive at a set $T_k$
	such that for all $v\in T_k\setminus T$, $c_G(T_k\setminus \{v\}) < c_G(T_k)$. This process will terminate because $|V_G|$ is finite. Moreover the following inequalities hold: $c_H(T)=c_G(T_0)\leq c_G(T_1)\leq \cdots \leq c_G(T_k)$.
	
	Let $G_1,\dots,G_{c_G(T_k)}$ be the connected components of $G \setminus T_k$, and suppose $H_j,H_{j'}\subset G_i$ for some $i$ and $j\neq j'$. Then at some point, a vertex must have been removed from a $T_\ell$ that connected $G_j$ and $G_{j'}$ in $G\setminus T_{\ell+1}$, where $H_j\subset G_j$, $H_{j'}\subset G_{j'}$, and $G_j$ and $G_{j'}$ are distinct connected components of $G\setminus T_\ell$. This shows that $c_G(T_{\ell+1})>c_G(T_\ell)$ contradicting the construction of $T_k$. Thus each $H_j$ is contained in a \textit{distinct} $G_i$.  
	
	By virtue of $T$ being an IDS of $H$, every $v\in T$ has $c_H(T\setminus \{v\})<c_H(T)$. That is, there are connected components $H_j$ and $H_{j'}$ such that the subgraph of $H$ induced by $V_{H_j}\cup V_{H_{j'}}\cup \{v\}$ is connected. If $H_j$ lies in $G_i$ and $H_{j'}$ lies in $G_{i'}$, with $i\neq i'$ as before, then the subgraph of $G$ induced by $V_{G_i}\cup V_{G_{i'}}\cup \{v\}$ is connected. Therefore, for all $v\in T$, $c_G(T_k\setminus \{v\})<c_G(T_k)$, and $T_k$ is indeed an IDS of $G$. 
	
	Let $P_G:=P_{T_k}$, our desired minimal prime. We will show $P_H \subset P_G$.
	For the first term of \eqref{eq:P_T}, since $T \subseteq T_k$,  $ (x_i,y_i \mid i \in T) \subseteq (x_i,y_i \mid i \in T_k).$ For the other terms, each $H_i$ is contained in some $G_j$ because $(T_k \setminus T) \cap V_H =
	\emptyset$. So $J_{\widetilde{H_i}} \subseteq J_{\widetilde{G_j}}$, and $P_H \subset P_G$.
	
	\textbf{Case 2: $I_G:=\gin(J_G)$}
	
	By \Cref{prop:new_gin_prime_decomp}, a minimal prime $P_H$ of $\gin(J_H)$ has the form
	$$ P_{T,U} = (x_i,y_i \mid i\in T)  + \sum_{i=1}^{c_H(T)} (x_i \mid i\in H_i, \, i \neq u _i)$$
	for a suitable IDS $T$ and set $U=\{u_1,\dots,u_{c_H(T)}\}$. Following the same method of Case 1, construct $T_k \supseteq T$. If $G_1,\dots,G_{c_G(T)}$ are the connected components of $G\setminus T_k$, as before, each $H_j\subseteq G_i$ for some $1\leq i\leq c_G(T_k)$. Relabel the $G_i$'s so that $H_j\subseteq G_j$ for $1\leq j\leq c_H(T)$, and let $U'=(u'_1,\dots,u'_{c_G(T_k)})$, where $u'_i=u_i$ for $1\leq i\leq c_H(T)$ and $u'_i\in V_{G_i}$ is chosen arbitrarily for $c_H(T)\leq i\leq c_G(T_k)$. Then $P_{T_k, U'}$ is our desired minimal prime.
	
	Since $T \subseteq T_k$, there is a containment $ (x_i,y_i \mid i \in T) \subseteq (x_i,y_i \mid i \in T_k).$ For the other terms, since $V_{H_i}\subset V_{G_i}$ and $u_i=u'_i$ for $1\leq i\leq c_H(T)$, there are containments $(x_i \mid i\in H_i, i\neq u_i)\subseteq (x_i\mid i\in G_i, i\neq u_i')$. Thus, $P_{T,U} \subset P_{T_k,U'}$.
	
	\textbf{Case 3: $I_G:= \inid(J_G)$}
	
	Take a minimal prime $Q_{T,U}$ of $\inid(J_H)$ as defined in \Cref{prop:in_prime_decomp}, and construct a minimal prime $Q_{T_k,U'}$ as in case 2. By the same reasoning, $Q_{T,U}\subset Q_{T_k,U'}$.
\end{proof}

It is not true that
every minimal prime $P_G$ contains exactly one minimal prime $P_H$, as
illustrated by the following counterexample on the net graph.

\begin{example}
	Let $G, H$ be as in \Cref{fig:net}. A minimal prime ideal of $\mathrm{in}_<(J_G)$ with $T=\{1,2,3\}$ is $P_T=(x_1,x_2,x_3,y_1,y_2,y_3)$, which contains all three of the minimal primes of $\mathrm{in}_<(J_H)$; $(x_1,x_2),(x_1,y_3),\text{and} \,(y_2,y_3)$. 
	
	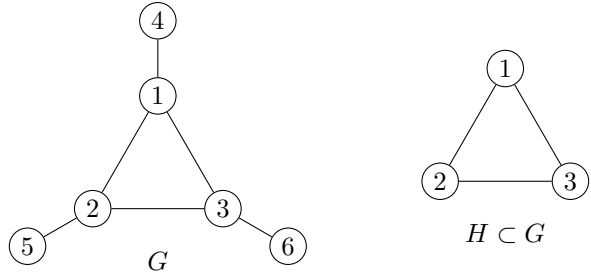
\begin{figure}[h!]
		\centering
		\begin{tikzpicture}[scale=1, every node/.style={circle, draw, inner sep=2pt}]
			\node (1) at (90:1) {$1$};
			\node (2) at (210:1) {$2$};
			\node (3) at (330:1) {$3$};
			\node (4) at (90:2) {$4$};
			\node (5) at (210:2) {$5$};
			\node (6) at (330:2) {$6$};
			\draw (1) -- (2) -- (3) -- (1);
			\draw (1) -- (4);
			\draw (2) -- (5);
			\draw (3) -- (6);
			\node[draw=none, inner sep=0pt] at (0,-1.2) {$G$};
		\end{tikzpicture}
		\qquad \qquad
		\begin{tikzpicture}[scale=1, every node/.style={circle, draw, inner sep=2pt}]
			\node (1) at (90:1) {$1$};
			\node (2) at (210:1) {$2$};
			\node (3) at (330:1) {$3$};
			\draw (1) -- (2) -- (3) -- (1);
			\node[draw=none, inner sep=0pt] at (0,-1.2) {$H\subset G$};
		\end{tikzpicture}
		\caption{The net graph and an induced subgraph}
		\label{fig:net}
	\end{figure}
\end{example}
With the prime decompositions established, we first compute the Waldschmidt constant of binomial edge ideals and their initial ideals.
\begin{thm}\label{thm: waldschmidt J_G}
	For all graphs $G$, $\widehat{\alpha}(J_G) = \widehat{\alpha}(\mathrm{in}(J_G)) = 2$.
\end{thm}
\begin{proof}
	Let $J_G$ be the binomial edge ideal of a graph $G$ with prime decomposition $J_G=\bigcap_{T\in \mathcal{T}(G)}P_T(G)$. Since $\widehat{\alpha}(J_G)\le \alpha(J_G)=2$ and $\widehat{\alpha}(\inid(J_G))\le \alpha(\inid(J_G))=2$, it suffices to show that $\widehat{\alpha}(J_G)$ and $\widehat{\alpha}(\inid(J_G))$ are each at least 2. 
	
	We first note that $J_{K_n}^{(m)}=J_{K_n}^{m}$ and $\inid(J_{K_n})^{(m)}=\inid(J_{K_n})^{m}$ by \cite[Corollary 3.4]{ene2011symbolic} and \cite[Theorem 5.9]{Simis1994OnTI}. Since $T=\emptyset$ is an IDS, if $|G|=n$ then the ideal $J_{K_n}$ is a minimal prime for $J_G$ and the minimal primes of $\inid(J_{K_n})$ are contained in the set of minimal primes for $\inid(J_G)$. In particular, $J_G^{(m)}\subset J_{K_n}^{(m)}$ and $\inid(J_G)^{(m)}\subset\inid(J_{K_n})^{(m)}$.
	Hence, 
	\begin{eqnarray*}
		\alpha(J_G^{(m)})\geq \alpha( J_{K_n}^{(m)}) &=&\alpha(J_{K_n}^m)=2m \ \text{ and }\\
		\alpha(\inid(J_G)^{(m)})\geq \alpha( \inid(J_{K_n})^{(m)}) &=&\alpha(\inid(J_{K_n})^m)=2m,
	\end{eqnarray*}
	whence the inequalities $\widehat{\alpha}(J_G) \geq  2$ and $\widehat{\alpha}(\inid(J_G)) \geq  2$ follow. Therefore $\widehat{\alpha}(J_G) = \widehat{\alpha}(\inid(J_G)) = 2$.
\end{proof}
Next we compute the asymptotic regularity for binomial edge ideals with equal ordinary and symbolic powers.
\begin{thm}
	\label{closed_graph_regularity}
	If $I$ is a homogeneous ideal in a polynomial ring $R$ that satisfies $I^{(m)}=I^m$ for all $m\ge1$, then $\widehat{\reg}(I)$ exists. If $I$ is also a monomial ideal or equigenerated, then $\widehat{\reg}(I)=d(I)$, where $d(I)$ is the maximal degree of a minimal generator of $I$.
\end{thm}
\begin{proof}
	By \cite[Theorem 1.1]{CHT}, if $I$ is a homogeneous ideal in a polynomial ring, then $\reg(I^m)$ is a linear function of $m$ for $m\gg0$, and $$
	\lim_{m\to\infty}\frac{\reg(I^m)}{m}=\lim_{m\to\infty} \frac{d(I^m)}{m},
	$$ where $d(I^m)$ is the maximal degree of a minimal generator of $I^m$. If $I$ is equigenerated, then every generator of $I^{m}$ has degree $m\cdot d(I)$. Therefore, $\widehat{\reg}(I)=\lim_{m\to\infty}d(I^m)/m=d(I)$. By \cite[Lemma 4.2]{dung21}, if $I$ is a monomial ideal and $x^{\alpha}\in I$ is a minimal generator of $I$, then $(x^{\alpha})^m$ is a minimal generator of $I^{(m)}=I^m$. In particular, $d(I^m)\ge m\cdot d(I)$. Since $I^m$ can have no generator with degree greater than $m\cdot d(I)$, it follows that $\widehat{\reg}(I)=\lim_{m\to\infty}d(I^m)/m=d(I)$.
\end{proof}

\Cref{closed_graph_regularity} allows us to find the asymptotic regularity of $J_G$ and $\inid(J_G)$ when $G$ is a very special type of graph:

\begin{defn}\label{def: closed graph and labeling}
	A graph $G$ is \textbf{closed} if there exists a labelling of the vertices such that, for all edges $\{i,j\}$ and $\{k,l\}$ with $i<j$ and $k<l$ one has $\{j,l\}\in E(G)$ if $i=k$, and $\{i,k\}\in E(G)$ if $j=l$. By \cite[Theorem 7.10]{HHO}, this condition is equivalent to a graph being chordal, claw-free, net-free, and tent-free.
\end{defn}

From now on, we assume $|G|=n$ and work over $S=\mathbb{K}[x_1,...,x_n,y_1,...,y_n]$, $\mathbb{K}$  a field. Then, the following corollary follows from \cite[Corollary 3.4]{ene2011symbolic} and \cite[Theorem 5.9]{Simis1994OnTI}, which establish the equality of ordinary and symbolic powers for $J_G$ and $\inid(J_G)$, respectively, provided $G$ is a closed graph.

\begin{cor}
	\label{closed_graph_cor}
	For a closed graph $G$, $\widehat{\reg}(J_G)=\widehat{\reg}(\inid(J_G))=2$.
\end{cor}

\Cref{closed_graph_cor} also follows from \cite[Prop. 3.2]{ene2016powers}, which states for closed graphs $G$ with connected components $G_1,G_2,...,G_{C(G)}$, we have
$$\reg(S/J_G^m)=\reg(S/\inid(J_G^m))=l_1+....+l_{C(G)}+2(m-1),$$
where $l_i$ is the length of the longest induced path of $G_i$. We note that \Cref{closed_graph_regularity} is slightly more general as it applies to all graphs satisfying $J_G^{(m)}=J_G^m$, including net-free caterpillar graphs \cite{Jahani} and pendant-clique graphs \cite{jahani2024}.

Jayanthan, Kumar, and Sankar showed in \cite[Proposition 3.3]{Jayanthan_2020} that if $H$ is an induced subgraph of $G$, then for all $i,j\ge0$ and $m\ge1$, $\beta_{ij}(S/J_H^m)\le\beta_{ij}(S/J_G^m)$ where $\beta_{i,j}(M)=\dim_\K \operatorname{Tor}_S(M,\K)$ denotes the $(i,j)-$th graded Betti number of a graded $S$-module $M$. We provide an analogous result for symbolic powers.

\begin{lem}
	\label{sympower_subgraph}
	Let $H$ be a subgraph of $G$ on the vertices $i_1,...,i_k$, and set $S_H = \K[x_{i_1},...,x_{i_k},y_{i_1},...,y_{i_k}]$. Then containments $J_H^{(m)}\subset J_G^{(m)}$ hold for all integers $m\geq 1$. If additionally $H$ is an induced subgraph, then $J_H^{(m)}=J_G^{(m)}\cap S_H$. 
\end{lem}
\begin{proof}
	Let $J_H=\bigcap_{T_H\in\mathcal{T}(H)}P_{T_H}$, $J_G=\bigcap_{T_G\in \mathcal{T}(G)}P_{T_G}$ be the prime decompositions of $J_H$ and $J_G$ respectively, as in \Cref{thm:beiprimedecomp}. Since $J_H\subset J_G$, by \Cref{lem:primedecompositioncontainment}, each prime ideal $P_{T_G}$ contains a prime ideal $P_{T_H}$. 
	
	If $f\in J_H^{(m)}$, then $f\in P_{T_H}^m$ for every $P_{T_H}$. Since ordinary powers preserve containment and every $P_{T_G}$ contains some $P_{T_H}$, $f\in P_{T_G}^m$ for every $P_{T_G}$. Therefore we have 
	$$J_H^{(m)}=\bigcap_{T_H\in \mathcal{T}(H)}P_{T_H}^m\subset\bigcap_{T_G\in \mathcal{T}(G)}P_{T_G}^m=J_G^{(m)}.$$
	
	Now suppose $H$ is an induced subgraph of $G$. By \cite[Proposition 3.3]{Jayanthan_2020}, there is an equality of prime decompositions
	$$\bigcap_{T_H\in \mathcal{T}(H)}P_{T_H}=J_H=J_G\cap S_H=\bigcap_{T_G\in \mathcal{T}(G)}(P_{T_G}\cap S_H).$$
	
	Consider the ring homomorphism $\pi:S\to S_H$ defined by $\pi(x_i)=\pi(y_i)=0$ if $i\not\in V_H$ and $\pi(x_i)=x_i$, $\pi(y_i)=y_i$ if $i\in V_H$. Note that if $p$ is a generator of $P_{T_G}$, then $\pi(p)=p$ if $p\in S_H$, and $\pi(p)=0$ if $p\not\in S_H$. By applying $\pi$ to $m$-fold products of elements of $P_{T_G}$, one deduces that $P_{T_G}^m\cap S_H=\pi(P_{T_G}^m\cap S_H)=(P_{T_G}\cap S_H)^m$.
	
	Hence, we obtain the containment
	$$J_G^{(m)}\cap S_H=\bigcap_{T_G\in \mathcal{T}(G)}(P_{T_G}^m\cap S_H)= \bigcap_{T_G\in \mathcal{T}(G)}(P_{T_G}\cap S_H)^m\subset \bigcap_{T_H\in \mathcal{T}(H)}P_{T_H}^m=J_H^{(m)},$$
	from which it follows that $J_H^{(m)}=J_G^{(m)}\cap S_H$.
\end{proof}

\begin{prop}
	\label{prop: betti subgraph}
	Let $H$ be an induced subgraph of $G$. Then for all $i,j>0$, $m\ge1$ the following inequality holds   $$\beta_{i,j}(S/J_H^{(m)})\le\beta_{i,j}(S/J_G^{(m)}).$$ In particular, $\reg(J_H^{(m)})\le \reg(J_G^{(m)})$ for every $m\ge1$ and $\widehat{\reg}(J_H)\le\widehat{\reg}(J_G)$.
\end{prop}
\begin{proof}
	Define $S_H$ as in \Cref{sympower_subgraph}. By \Cref{sympower_subgraph}, there is a natural inclusion $i:S_H/J_H^{(m)}\hookrightarrow S/J_G^{(m)}$ given by $i(r+J_H^{(m)})=r+J_G^{(m)}$. Let $\overline{\pi}:S/J_G^{(m)}\to S_H/J_H^{(m)}$ be the map induced by the ring homomorphism $\pi$ defined in the proof of \Cref{sympower_subgraph}. To show $\overline{\pi}$ is well-defined, it suffices to show that if $r\in J_G^{(m)}$, then $\pi(r)\in J_H^{(m)}$. This follows since $\pi(J_G^{(m)})=J_G^{(m)}\cap S_H=J_H^{(m)}$ by \Cref{sympower_subgraph}.
	
	Finally, we observe that $\overline{\pi}\circ i$ is the identity on $S_H/J_H^{(m)}$, so $S_H/J_H^{(m)}$ is an algebra retract of $S/J_G^{(m)}$. The claim regarding Betti numbers then follows from \cite[Corollary 2.5]{OHH} and the remaining claims follow from this claim.
\end{proof}

The preceding result allows us to show a lower bound on the regularity of symbolic powers of binomial edge ideals analogous to \cite[Corollary 3.4]{Jayanthan_2020}:

\begin{cor}
	For all graphs $G$, $reg(J_G^{(m)})\ge 2m+l(G)-2$ where $l(G)$ is the length of the longest induced path in $G$.
\end{cor}

\begin{proof}
	This follows from \Cref{prop: betti subgraph}, \cite[Observation 3.2]{Jayanthan_2020}, and \cite[Corollary 3.4]{ene2011symbolic}.
\end{proof}

\section{Symbolic polyhedra}

We turn our focus to the asymptotic invariants associated with the relevant monomial ideals of $J_G$, emphasizing $\gin(J_G)$ due to its convenient structure. \Cref{thm:gin_gen} shows that permuting the labeling on the vertices of $G$ corresponds to permuting the variables of $\gin{J_G}$. Our main tool in this section is the symbolic polyhedron. To use it effectively, we analyze its vertices in a general setting. Herein, we let $I$ be a squarefree monomial ideal.
\subsection{Vertices of the symbolic polyhedron}
We start with a criterion for detecting vertices.
\begin{lem}
	\label{lem:vertex_criterion}
	If  $I\subset \mathbb{K}[x_1,\dots x_n]$ and $a\in \mathcal{V}(\SP(I))$, then $a$ lies on at least $n$ facets of $\SP(I)$.
\end{lem}
\begin{proof}
	Suppose $a$ lies on exactly $m<n$ facets of $\SP(I)$. Each facet defines a hyperplane in $\mathbb{R}^n$. The intersection of $m$ hyperplanes is a linear subspace of $\R^n$ of dimension at least  $n-m\geq 1$. Since $a$ lies on no further facets by assumption, $a$ lies in the interior of a face of $\SP(I)$ with dimension at least one. Therefore, $a$ is not a vertex.
\end{proof}
In several proofs, we will invoke this lemma in conjunction with the following property of facets.
\begin{obs}
	\label{obs:facets}
	For every non-zero vector $\textbf{u}$ and $\epsilon>0$, if $a\in \SP(I)$ lies on a one-facet $F_{M}$, that is, $\sum_{i\in M}a_i=0$ and $a+\epsilon\textbf{u},a-\epsilon\textbf{u}\in \SP(I)$, both $a+\epsilon\textbf{u}$ and $a-\epsilon\textbf{u}$ lie on $F_{M}$. Otherwise, 
	$$\sum_{i\in M}(a\pm \epsilon\textbf{u})_i >1\implies \sum_{i\in M}(a\mp \epsilon\textbf{u})_i <1,$$ 
	contradicting that $a\mp\epsilon\textbf{u}\in \SP(I)$.
\end{obs}

We give the following characterization of vertices of $\SP(I)$ whose entries are only 0's and 1's.
\begin{lem}
	\label{lem:0/1_vertex_criterion}
	Take $a\in \SP(I)$ with $a_i=\begin{cases}
		1, & i\in K \\
		0, & i\notin K
	\end{cases}$, $K=(i_1,\dots, i_k)\subset [n]$. Then $a\in \mathcal{V}(\SP(I))$ if and only if there are $k$ one-facets $F_{M_j}$ such that for $1\leq j\leq k$, $M_j\cap K=\{i_j\}$.
\end{lem}
\begin{proof}
	$(\Rightarrow)$ If $a\in \mathcal{V}(\SP(I))$, \Cref{lem:vertex_criterion} implies $a$ lies on at least $k$ one-facets. If there is $i\in K$ with $\{i\}\neq K\cap M$ for all minimal primes $P_M$ of $I$, then for each minimal prime $P_M$, $i\in K\cap M$ implies that there exists $i'\neq i$ such that  $i'\in K\cap M$. If $i\in M$, then for $0<\epsilon\leq1$ we have 
	$$
	\sum_{j\in M}(a \pm \epsilon \textbf{e}_i)_j \geq a_i-1+a_{i'}=1,
	$$
	which implies that $a\pm \epsilon \textbf{e}_i\in \SP(I)$. This contradicts the assumption that $a$ is a vertex in view of the last sentence in \Cref{defn:facets}.
	
	$(\Leftarrow)$ Assume such $F_{M_j}$ exist. Take $\textbf{u}\in \mathbb{R}^n$ and suppose $a\pm \epsilon \textbf{u}\in \SP(I)$ for some $\epsilon>0$. By \Cref{obs:facets}, since $a$ lies on $F_{M_j}$, $a\pm \epsilon \textbf{u}$ lies on $F_{M_j}$ and thus $a_{i_j}\pm \epsilon \textbf{u}_{i_j}=1$. So, $\textbf{u}_{i_j}=0$ for every $i_j\in K$.
	If $|\textbf{u}_j|>0$ for some $j\in [n]\backslash K$, either $a_j+\epsilon\textbf{u}_j<0$ and $a+ \epsilon\textbf{u}\notin \SP(I)$, or $a_j-\epsilon\textbf{u}_j<0$ and $a- \epsilon\textbf{u}\notin \SP(I)$. Therefore $\textbf{u}=0$, and per \Cref{defn:facets}, we conclude $a\in \mathcal{V}(\SP(I))$.
\end{proof}

The next result shows that the vertices of the symbolic polyhedron $\SP(I)$  which have only 0 and 1 coordinates correspond to minimal generators of $I$.

\begin{prop}
	\label{prop:0/1_vertices}
	Take $a\in \SP(I)$ with $a_i=\begin{cases}
		1, & i\in K \\
		0, & i\notin K
	\end{cases}$, $K=(i_1,\dots, i_k)\subset [n]$. Then $a\in\mathcal{V}(\SP(I))$ if and only if $ g=\prod_{i\in K}x_i$ is a minimal generator of $I$.
\end{prop}
\begin{proof}  
	$(\Rightarrow)$ 
	For every minimal prime $P_M$ of $I$, since $g\in I\subset P_M$, $x_i\in P_M$ for some $i\in K$. Also, if there were $i\in K$ such that $x_i\notin P_M$ for every $P_M$ of $I$, $g$ would not be a minimal generator. Furthermore, suppose there were $i\in K$ such that for every minimal prime $P_M$ of $I$, $x_i\in P_M$ implies $x_{i'}\in P_M$ for some $i'\in K\setminus \{i\}$. Then $g'=\frac{g}{x_i}\in I$, contradicting $g$ being a minimal generator. We can therefore pick minimal primes $P_{M_1},\dots, P_{M_k}$ of $I$ with $M_{j}\cap K=\{i_j\}$. The corresponding facets imply $a\in \mathcal{V}(\SP(I))$ by \Cref{lem:0/1_vertex_criterion}. 
	
	$(\Leftarrow)$ Since $a\in \SP(I)$, every one-facet $F_M$ has $M$ containing some $i\in K$. Then $x_i\in P_M$, so $g=\prod_{i\in K}x_i\in \bigcap P_M=I$. Define one-facets $F_{M_1},\dots F_{M_k}$ as in \Cref{lem:0/1_vertex_criterion} and their corresponding minimal primes $P_{M_1},\dots, P_{M_k}$. For any $1\leq j\leq k$, $ g'_{i_j}=\frac{g}{x_{i_j}}\notin P_{M_j}\supset I$. Thus, no monomial properly dividing $g$ is in $I$, so $g$ is a minimal generator of $I$.
\end{proof}
In fact, every vertex of the symbolic polyhedron corresponds to a generator of the corresponding symbolic power of the ideal, determined by its denominator.
\begin{prop}\label{prop: vsp}
	Let $a=\frac{1}{q}(z_1,\dots,z_n)\in \mathcal{V}(\SP(I))$, where $q,z_1,\dots,z_n\in \mathbb{N}$. Then $g=x_1^{z_1}\cdots x_n^{z_n}$ is a minimal generator of $I^{(q)}$.
\end{prop}
\begin{proof}
	By \cite[Theorem 5.4]{ha_newton-okounkov_2021}, there is an integer $c$ for which $\NP(I^{(c)})=c\SP(I)$, with $q\mid c$. Then if $d=\frac{c}{q}$,  $x_1^{dz_1}\cdots x_n^{dz_n}\in I^{(c)}$, implying for every minimal prime $P_M=(x_i\mid i\in M)$ of $I$ that $x_1^{dz_1}\cdots x_n^{dz_n}\in P_M^{(c)}$. Since $P_M$ is a monomial prime ideal, thus generated by a regular sequence of variables, $P_M^{(m)}=P_M^m$ for all $m\geq 1$. 
	Then $x_1^{dz_1}\cdots x_n^{dz_n}\in P_M^c$, so $\prod_{i\in M}x_i^{dz_i}\in P_M^c$. This means that $\sum_{i\in M}dz_i\geq c$, so $\sum_{i\in M}z_i\geq q$ and $\prod_{i\in M}x_i^{z_i}\in P_M^q=P_M^{(q)}$. We conclude $g=x_1^{z_1}\cdots x_n^{z_n}\in I^{(q)}$. 
	
	If $g$ were not a minimal generator of $I^{(q)}$, then there is $g'=x_1^{z_1'}\cdots x_n^{z_n'}\in I^{(q)}$, with $g'\neq g$ and $z_i'\leq z_i$ for $1\leq i\leq n$. Then if $a'=\frac{1}{q}(z'_1,\dots,z_n')$, $a'\in\frac{\NP(I^{(q)})}{q}\subset \SP(I)$ and $a-a'\in \mathbb{R}^n_{\geq 0}\setminus \{0\}$, so $a+(a-a')\in \SP(I)$. Since $a-(a-a')=a'\in \SP(I)$, this contradicts $a\in \mathcal{V}(\SP(I))$, so $g$ is indeed minimal.    
\end{proof}
We demonstrate these propositions in the following example. Vertices will either be denoted as column vectors or points, as necessary for compactness.
\begin{example}
	Let $I=(x_1x_2x_3,x_1x_2x_4,x_3x_4)$. The symbolic polyhedron of $I$ has vertices
	$$\mathcal{V}(\SP(I))=\left\{\begin{bmatrix}
		1 \\ 1 \\ 1 \\ 0
	\end{bmatrix},\begin{bmatrix}
		1 \\ 1 \\ 0 \\ 1
	\end{bmatrix},\begin{bmatrix}
		0 \\ 0 \\ 1 \\ 1
	\end{bmatrix},\tfrac{1}{2}\begin{bmatrix}
		1 \\ 1 \\ 1\\1
	\end{bmatrix}\right\}.$$
	The second symbolic power of $I$ is given by
	$$I^{(2)}=(x_1^2x_2^2x_3^2,x_1^2x_2^2x_4^2,x_1x_2x_3x_4, x_3^2x_4^2). $$
	The four vertices correspond to the three generators of $I$ and the generator of $I^{(2)}$ not in $I^2$.
\end{example}
As noted in \cite[Remark 6.7]{ha_newton-okounkov_2021}, the converse of \Cref{prop: vsp} does not hold in general. By \Cref{prop:0/1_vertices}, it holds for $q=1$, allowing us to infer the following lower bound on $\areg(\gin(J_G))$
\begin{cor}
	\label{cor:induced_path_lower_bound}
	For a graph $G$, let $\ell$ be the length of the longest induced path in $G$. Then 
	$\areg(\gin(J_G))\geq \ell.$
\end{cor}
\begin{proof}
	If $P$, a path with vertex set $K=\{i,v_1,\dots,v_s,j\}$, is the longest induced path in $G$, then $y_{v_1}\cdots y_{v_s} x_i x_j$ is a minimal generator of $\gin(J_G)$ by \Cref{thm:gin_gen}. By \Cref{prop:0/1_vertices}, $a\in \mathcal{V}(\SP(\gin(J_G)))$, where $a_i=a_j=a_{n+v_1}=\dots=a_{n+v_s}=1$ and for all other $1\leq k\leq 2n$, $a_k=0$. Then $$\areg(\gin(J_G))\geq \sum_{i=1}^{2n}a_i=|K|=\ell.$$
\end{proof}
It is known by \cite[Lemma 4.2]{dung21} and \cite[Theorem 1.1]{CHT} that the asymptotic regularity of any monomial ideal is bounded by the maximal degree of a minimal generator. This result also proves the corollary.

\subsection{Decomposing the symbolic polyhedra} Assuming hereafter that $|G|=n$, we now focus on the symbolic polyhedra of $\inid(J_G)\subset S$ and $\gin(J_G)\subset S$, where $S=\mathbb{K}[x_1,\dots,x_n,y_1,\dots,y_n]$. We will denote minimal primes of $I$ by $P_M$, with $M\subset [2n]$, where 
$$
P_M=\{x_i\mid i\in M, 1\leq i\leq n\} \cup \{y_{i-n}\mid i\in M, n+1\leq i\leq 2n\}.
$$
We prove two facts that will give rise to a partition of the vertices of the polyhedra in terms of a graph's induced connected subgraphs. 

First, we consider $G$ and its induced subgraphs. The following theorem gives a geometric substitute for the algebra retract map relating the symbolic powers of $I_G$ and $I_H$ in \Cref{prop: betti subgraph}.

\begin{prop}
	\label{thm:vertex_containment}
	Let $H$ be an induced subgraph of $G$. For both $I_G=\inid (J_G)$ and $I_G=\gin(J_G)$, under the canonical inclusion $i:\mathbb{R}^{2|H|}\rightarrow \mathbb{R}^{2n}$,
	$$ i(\mathcal{V}(\SP(I_H))) \subseteq \mathcal{V}(\SP(I_G)).$$
\end{prop}
\begin{proof}
	
	Since for $a\in \SP(I_G)$ we have $a_i\geq 0$, it follows that all points $a\in \SP(I_G)$ satisfy the inequality $\sum_{j\in [n]\setminus V_H}(a_j+a_{n+j})\geq 0$. Consider the face
	$$ F=\SP(I_G)\cap \left\{a \in\R^{2n} \;\middle|\; \sum_{j\in [n]\setminus V_H}(a_j+a_{n+j}) =0\right\}.$$
	Let $V_H\subset [n]$ be the vertices of $H$. We may index elements of $\mathbb{R}^{2|H|}$ by $V_H$ and $V_H+n$ instead of $[2|H|]$ in the natural way. The inclusion $i$ is then defined for $1\leq j\leq n$ as $$(i(a))_j=\begin{cases}
		a_j, & j\in V_H\\
		0,  & j\notin V_H
	\end{cases}\quad \text{and}\quad (i(a))_{n+j}=\begin{cases}
		a_{n+j}, & j\in V_H\\
		0,  & j\notin V_H
	\end{cases}.$$ We show $F=i(\SP(I_H))$.
	
	The set $F$ can be equivalently defined $$F=\left\{a\in \SP(I_G)\mid a_j=a_{n+j}=0 \text{ for } j \in [n]\setminus V_H\right\}.$$ Thus, if we prove $i(\SP(I_H))\subset \SP(I_G)$, $i(\SP(I_H))\subset F$ follows. 
	
	We show that for any $a\in \mathrm{\SP}(I_H)$, every facet inequality of $\SP(I_G)$ is satisfied by $i(a)$. First, the nonnegative orthant of $\R^{2|H|}$ is mapped to the nonnegative orthant of $\R^{2n}$, so all zero-facet inequalities are satisfied by $i(a)$. Next, consider a minimal prime $P_M$ of $I_G$. By \Cref{lem:primedecompositioncontainment}, $P_M$ contains some minimal prime $P_{M'}$ of $I_H$. Since $a\in \SP(I_H)$, the defining inequality of $F_{M'}$  is satisfied by $a$. Therefore, $M'\subset M$ implies $F_M$'s inequality is satisfied by $i(a)$. So, $i(a)\in \SP(I_G)$ and we have established $i(\SP(I_H))\subset \SP(I_G)$, implying $i(\SP(I_H))\subset F$. 
	
	To show $F\subset i(\SP(I_H))$, we show that for any $a\in F$, the facet inequalities of
	$\SP(I_H)$ are satisfied by $\pi(a)$, where $\pi: \mathbb{R}^{2n}\rightarrow \mathbb{R}^{2|H|}$ is the
	canonical projection. Under $\pi$, the nonnegative orthant of $\R^{2n}$ is mapped to the nonnegative orthant of $\R^{2|H|}$, so all zero-facet inequalities are satisfied by $\pi(a)$. Take minimal prime $P_{M'}$ of $I_H$.
	\Cref{lem:primedecompcontainment} says that $P_{M'}$ is contained in
	some minimal prime $P_{M}$, with $M'= M\cap \{j,n+j\mid j\in V_H\}$. Since $a\in \SP(I_G)$, $F_M$'s inequality is satisfied by $a$. However, since $a\in F$, $a_{j}=a_{n+j}=0$ for $j\in [n]\setminus J_H$. Thus, $\pi(a)$ satisfies the inequality of $F_{M'}$, and $\pi(a)\in \SP(I_H)$. Then $\pi(F)\subset \SP(I_H)$, implying $i(\pi(F))=F\subset i(\SP(I_H))$. 
\end{proof}

Secondly, we consider $G$ and its connected components.

\begin{thm}
	\label{thm:disconnected_vertices}
	Suppose $G$ has connected components $G_1,\dots G_s$.  For both $I_G=\gin({J_G})$ and $I_G=\inid({J_G})$, under canonical inclusions $i_j:\mathbb{R}^{2|G_j|}\rightarrow \mathbb{R}^{2n}$ analogous to that in \Cref{thm:vertex_containment}, the vertices of the symbolic polyhedron $\SP(I_G)$ can be partitioned as
	$$\mathcal{V}(\SP(I_G))=\bigsqcup_{j=1}^s i_j(\mathcal{V}(\SP(I_{G_j}))).$$
\end{thm}
\begin{proof}
	One containment is immediate by \Cref{thm:vertex_containment}. To show $\mathcal{V}({\SP(I_G))}\subset \bigcup_{j=1}^s i_j(\mathcal{V}(\SP(I_{G_j})))$, note that the generators of $I_{G_1},\dots,I_{G_s}$ have distinct sets of variables per \Cref{thm:inid_gen} or \Cref{thm:gin_gen}. Then $I_G=I_{G_1}+\dots+I_{G_s}$, and by \cite[Theorem 4.1]{ha2021}, the symbolic powers of $I_G$ decompose as
	\begin{equation}\label{eq: binomial}
		I_G^{(r)}=\sum_{b_1+\dots+b_s=r}{I_{G_1}}^{(b_1)}\cdots{I_{G_s}}^{(b_s)}.
	\end{equation}
	Take $a\in \mathcal{V}(\SP(I_G))$.  By \Cref{defn:newton_polyhedron}, there exists a positive integer $q$ such that $a=\frac{1}{q}(z_1,\dots,z_{2n})$ and $qa=(z_1,\dots,z_{2n})\in NP(I_G^{(q)})$. By \eqref{eq: binomial}, there are positive integers $b_1,\dots, b_s$ and $m_j\in NP(I_{G_j}^{(b_j)})$ with $\sum_{j=1}^sb_j=q$ and $\sum_{j=1}^s i_{j}(m_j)=(z_1,\dots,z_{2n})$. Let $m_{t_1},\dots,m_{t_r}$ be the nonzero $m_i$'s, with $\frac{\ m_{t_j}}{b_{t_j}}\in
	\SP(I_{G_{t_j}}).$ Since $ \sum_{j=1}^rb_{t_j}=q$ gives $\sum_{j=1}^r\frac{b_{t_j}}{q}=1$, we have the convex combination
	$$\frac{b_{t_1}}{q}i_{t_1}\left(\frac{
		m_{t_1}}{b_{t_1}}\right)+\dots+\frac{b_{t_r}}{q}i_{t_r}\left(\frac{
		m_{t_r}}{b_{t_r}}\right)=\frac{i_{t_1}( m_1)+\dots+i_{t_r}(
		m_{t_r})}{q}=\frac{(z_1,\dots,z_{2n})}{q}=a,$$ with $$
	i_{t_j}\left(\frac{m_{t_j}}{b_{t_j}}\right)\in
	i_{t_j}(\SP(I_{G_{t_j}}))\subset \SP(I_{G})$$ by
	\Cref{thm:vertex_containment}. Since $a\in \mathcal{V}(\SP(I_G))$, it cannot be written as a
	non-trivial convex combination of other points in $\SP(I_G)$. Therefore, $r=1$ and $a\in i_{t_1}(\mathcal{V}(\SP(I_{G_{t_1}})))$.
\end{proof}
These theorems allow us to parse the vertices of the symbolic polyhedron by a graph's induced connected subgraphs. 
\begin{cor} 
	\label{cor:connected_Va}
	Take $a\in \mathcal{V}(\SP(I_G))$, where $I_G=\inid(J_G)$ or $I_G=\gin(J_G)$. Define $$V_{a}= \{j\in [n] \mid a_j+a_{n+j} \neq 0\}.$$ Then the subgraph $H_a$ of $G$ induced by $V_{a}$ is connected, and under the canonical inclusion $i:\mathbb{R}^{2|H_a|}\rightarrow \mathbb{R}^{2n}$, $a\in i(\SP(I_{H_a}))$. \end{cor} \begin{proof} 
	Consider the face $$F=\SP(I_G)\cap \left\{a\in\mathbb{R}^{2n}
	\;\middle|\; \sum_{j\in [n]\setminus V_a} a_j+a_{n+j}=0\right\}.$$ By
	\Cref{thm:vertex_containment}, $F=i(\SP(I_{H_a}))$, and as noted in \Cref{defn:facets}, $a$ is a vertex of $F$. Suppose $H_a$ were disconnected, with connected component $H_{a'}$ and vertex set $V_{a'}$.
	Then by \Cref{thm:disconnected_vertices}, either $a_j=a_{n+j}=0$ for all
	$j\in V_{a'}$ or for all $j\in V_a\backslash V_{a'}$. In either case,
	this contradicts the construction of $V_a$. The corollary follows.
\end{proof}
With this corollary, it makes sense to define the vertices with $V_a=[n]$.
\begin{defn}
	\label{def:full_vertices}
	If $G$ is a connected graph on $[n]$, let
	$$\mathcal{V}_F(\SP(I_G))=\mathcal{V}(\SP(I_G))\cap \left\{a\in \mathbb{R}^{2n} \;\middle|\; a_j+a_{n+j}\neq 0\text{ for } 1\leq j\leq n \right\}.$$ We call these the \textbf{full vertices} of $G$.
\end{defn}
As we shall see in section \ref{s: waldschmidt gin} full vertices can be easier to determine than arbitrary vertices of a symbolic polyhedron.

\begin{thm} 
	\label{cor:subgraph_decomp} 
	Let $G$ be a connected
	graph, and let $\mathcal{H}_G$ denote the set of connected induced subgraphs of
	$G$ with at least one edge. With $I_H=\mathrm{gin}(J_H)$ or $I_H=\inid (J_H)$
	for $H\in \mathcal{H}_G$, under canonical projections $i_H:\mathbb{R}^{2|H|}\rightarrow \mathbb{R}^{2n}$, the set of vertices of $\SP(I_G)$ can be recovered from the sets of full vertices of graphs in $\mathcal{H}_G$ as follows
	$$\mathcal{V}(\SP(I_G))=\bigcup_{H\in
		\mathcal{H}_G} i_H(\mathcal{V}_F(\SP(I_H))).$$ 
	Furthermore, \begin{align*}
		\wald(I_G) &=\min\left\{\min\left\{\sum_{i=1}^{2n}a_i\;\middle|\; a\in \mathcal{V}_F(\SP(I_H))\right\}\;\middle|\;H\in\mathcal{H}_G\right\} \\
		\widehat{\operatorname{reg}}(I_G) &=\max\left\{\max\left\{\sum_{i=1}^{2n}a_i\;\middle|\; a\in \mathcal{V}_F(\SP(I_H))\right\}\;\middle|\;H\in\mathcal{H}_G\right\}
	\end{align*}
\end{thm} 
\begin{proof}
	This follows immediately from \Cref{thm:vertex_containment}, \Cref{thm:disconnected_vertices}, \Cref{cor:connected_Va}, and \Cref{thm:wald&areg_via_SP}.
\end{proof}
We illustrate this decomposition in the following example:
\begin{example}
	Let $P_4$ be the path graph on 4 vertices. Then $\mathcal{V}(\SP(\gin(J_{P_4})))$ is
	$$\left\{\begin{bmatrix}
		1  \\ 1  \\ 0 \\  0  \\ 0\\ 0  \\ 0   \\0 
	\end{bmatrix},\begin{bmatrix}
		0 \\ 1  \\ 1 \\  0  \\ 0\\ 0  \\ 0   \\0 
	\end{bmatrix},\begin{bmatrix}
		0 \\ 0  \\ 1 \\  1  \\ 0\\ 0  \\ 0   \\0 
	\end{bmatrix},\begin{bmatrix}
		1  \\ 0  \\ 1 \\  0  \\ 0\\ 1  \\ 0   \\0 
	\end{bmatrix},\begin{bmatrix}
		0 \\ 1  \\ 0  \\ 1  \\ 0\\ 0  \\ 1   \\0 
	\end{bmatrix},\begin{bmatrix}
		1 \\ 0  \\ 0  \\ 1  \\ 0\\ 1  \\ 1   \\0 
	\end{bmatrix},\tfrac{1}{2}\begin{bmatrix}
		1 \\ 1 \\ 1 \\ 0 \\ 0 \\ 1 \\ 0 \\ 0
	\end{bmatrix},\tfrac{1}{2}\begin{bmatrix}
		0 \\ 1 \\ 1 \\ 1 \\ 0 \\ 0 \\ 1 \\ 0
	\end{bmatrix},\tfrac{1}{2}\begin{bmatrix}
		1 \\ 1 \\ 0 \\ 1 \\ 0 \\ 1 \\ 1 \\ 0
	\end{bmatrix},\tfrac{1}{2}\begin{bmatrix}
		1 \\ 0 \\ 1 \\ 1 \\ 0 \\ 1 \\ 1 \\ 0
	\end{bmatrix},\tfrac{1}{3}\begin{bmatrix}
		1 \\ 1 \\ 1 \\ 1 \\ 0 \\ 1 \\ 1 \\ 0
	\end{bmatrix}\right\}.$$
	As defined in \Cref{cor:subgraph_decomp}, $\mathcal{H}$ has $P_4$, its three edges isomorphic to $P_2$, and its two subpaths isomorphic to $P_3$. The full vertices of $P_2$, $P_3$, and $P_4$ are as follows:
	\begin{flalign*}
		&\mathcal{V}_F(\SP(\gin(J_{P_2})))=\left\{\begin{bmatrix}
			1 \\ 1 \\ 0 \\ 0
		\end{bmatrix}\right\}, \mathcal{V}_F(\SP(\gin(J_{P_3})))=\left\{\begin{bmatrix}
			1 \\ 0 \\ 1 \\ 0 \\ 1 \\ 0 \\
		\end{bmatrix},\tfrac{1}{2}\begin{bmatrix}
			1 \\ 1 \\ 1 \\ 0 \\ 1 \\ 0 \\
		\end{bmatrix}\right\}
		\\
		&\mathcal{V}_F(\SP(\gin(J_{P_4})))=\left\{\begin{bmatrix}
			1 \\ 0 \\ 0 \\ 1 \\ 0 \\ 1 \\ 1 \\ 0 \\
		\end{bmatrix},\tfrac{1}{2}\begin{bmatrix}
			1 \\ 1 \\ 0 \\ 1 \\ 0 \\ 1 \\ 1 \\ 0 \\
		\end{bmatrix},\tfrac{1}{2}\begin{bmatrix}
			1 \\ 0 \\ 1 \\ 1 \\ 0 \\ 1 \\ 1 \\ 0 \\
		\end{bmatrix},\tfrac{1}{3}\begin{bmatrix}
			1 \\ 1 \\ 1 \\ 1 \\ 0 \\ 1 \\ 1 \\ 0 \\
		\end{bmatrix}\right\}
	\end{flalign*}
	Under suitable inclusions, we recover the 11 vertices in $\mathcal{V}(\SP(\gin(J_{P_4})))$.
\end{example}

\section{The Waldschmidt constant of the multigraded generic initial ideal}\label{s: waldschmidt gin}

We have now developed the necessary tools to study $\wald(\gin(J_G))$. The edge ideal is the precursor of the binomial edge ideal, defined as follows.
\begin{defn}
	For a simple graph $G$ with edge set $E_G$, the edge ideal $I_G\subset \mathbb{K}[x_1,\dots,x_n]$ is defined by $$I_G=(x_ax_b\mid \{a,b\}\in E_G).$$
	
\end{defn}
For our purposes, we will consider $I_G$ as an ideal of $K[x_1,\dots,x_n,y_1,\dots,y_n]$, with $\SP(I_G)\subset \mathbb{R}_{\geq 0}^{2n}$. Edge ideals have convenient prime decompositions.
\begin{lem}\cite[Corollary 1.35]{tuyl2023}
	The minimal primes of $I_G$ are in one-to-one correspondence with the minimal vertex covers of $G$.
\end{lem}
An important observation is that for a simple graph $G$, $I_G\subset \gin(J_G)$, since the edges of $G$ are induced paths. We show that the Waldschmidt constant of these ideals is equal. The main idea is that by the structure of the gin the Waldschmidt constant can be attained at a point whose $y$-coordinates are all zero.

\begin{thm}
	\label{thm:wald_equality}
	For a simple graph $G$, there is an equality $$\wald(\gin(J_G))=\wald(I_G).$$
\end{thm}
\begin{proof}
	Since $I_G\subset \gin(J_G)$, $I_G^{(q)}\subset \gin(J_G)^{(q)}$ for every $q\geq 1$. Since both ideals are squarefree monomial ideals $\alpha(I_G^{(q)})\geq \alpha(\gin(J_G)^{(q)})$, so $\wald(I_G)\geq \wald(\gin(J_G))$ by \cite[Lemma 3.10]{DiPasquale19}.
	
	To prove that $\wald(I_G)\leq \wald(\gin(J_G))$, we first show if $a\in \SP(\gin(J_G))$ and $a_{n+1}=\dots=a_{2n}=0$, then $a\in \SP(I_G)$. Let $C$ be a minimal vertex cover of $G$, and take the minimal prime $P_C=(x_a\mid a\in C)$ of $I_G$. We construct algorithmically an IDS for $G$ which witnesses that the ideal $P_C+(y_i\mid n+i\in M)$  is a minimal prime  of $\gin(J_G)$. Initialize  $T_{C,0}=C$. Since $C$ is a minimal vertex cover, $V_G\setminus C$ is a maximal independent set, so $c_G(T_0)\geq c_G(S)$ for any $S\subset T_0$. If there is $v\in T_0$ such that $c_G(T_0)=c_G(T\setminus \{v\})$, let $T_{C,1}=T_{C,0}\setminus\{v\}$, and repeat the above process for $T_{C,1}$. This algorithm will terminate since $G$ is finite; call its output $T_{C,\Omega}$. By construction, $T_{C,\Omega}$ is an IDS for $G$, and each connected component of $G\setminus T_{C,\Omega}$ contains exactly one vertex not in $C$. Then, in the notation of \Cref{prop:new_gin_prime_decomp}, $P_{T_{C,\Omega}, [n]\setminus C}$ is a minimal prime of $\gin(J_G)$, with $$P_{T_{C,\Omega}, [n]\setminus C}=(x_i,y_i\mid i\in T_{C,\Omega})+(x_i\mid i\in C\setminus T_{C,\Omega})=P_C+(y_i\mid i\in T_{C,\Omega}).$$ Thus, if $a\in \SP(\gin(J_G))$ and $a_{n+1}=\dots=a_{2n}=0$, then for every minimal vertex cover $C$,
	$$\sum_{i\in C}a_i=\sum_{i\in C}a_i+\sum_{i\in T_{C,\Omega}}a_{n+i}\geq 1,$$ since $a$ satisfies the corresponding facet inequality of $P_{T_{C,\Omega}, [n]\setminus C}$. As desired, we conclude that $a\in \SP(I_G)$.
	
	Now, we show that there is $a\in \SP(\gin(J_G))$ with $a_{n+1}=\dots=a_{2n}=0$ and $\sum_{i=1}^{2n}a_i=\wald(\gin(J_G))$. For $b\in \SP(\gin(J_G))$, define $b'$  by $b'_i=\begin{cases}
		b_i+b_{n+i}, & i\leq n \\
		0, & i\geq n+1
	\end{cases}.$
	Observe that $\sum_{i=1}^{2n}b_i'=\sum_{i=1}^{2n}b_i$. For any minimal prime $P_M$ of $\gin(J_G)$, if $n+i\in M$, then $i\in M$. We can then split
	$$\sum_{i\in M}b_i=\sum_{\{i\in M\cap [n]|n+i\in M\}} \left(b_i+b_{n+i}\right)+\sum_{\{i\in M\cap [n]|n+i\not\in M\}} b_i.$$ Then since $b_{j}'
	=0$ for $j\geq n+1$, and $\sum_{i\in M}b_i\geq 1$, we have
	\begin{eqnarray*}
		\sum_{i\in M}b'_i &=& \sum_{\substack{i\in M\cap [n] \\ n+i\in M}} b_i'+\sum_{\substack{i\in M\cap [n] \\ n+i\notin M}} b_i
		'=\sum_{\substack{i\in M\cap [n] \\ n+i\in M}} (b_i+b_{n+i}) +\sum_{\substack{i\in M\cap [n] \\ n+i\notin M}} (b_i+b_{n+i}) \\&=&\sum_{i\in M} b_i+ \sum_{\substack{i\in M\cap [n] \\ n+i\notin M}} b_{n+i}\geq 1 \qquad \qquad \qquad  \text{since }b\in\SP(\gin(J_G)).
	\end{eqnarray*}
	We conclude $b'\in \SP(\gin(J_G))$. Therefore,  taking $a\in \SP(\gin(J_G))$ with $\sum_{i=1}^{2n}a_i=\wald(\gin(J_G))$, $a'\in \SP(\gin(J_G))$ has $a'_{n+1}=\dots=a'_{2n}=0$ and $\sum_{i=1}^{2n}a'_i=\wald(\gin(J_G))$. Our previous result shows $a'\in \SP(I_G)$, so $\wald(I_G)\leq \wald(\gin(J_G))$ and the result follows.
\end{proof}

While $\SP(\gin(J_G))$ contains $\SP(I_G)\subset \mathbb{K}[x_1,\dots,x_n,y_1,\dots,y_n]$, the two are not always equal. In fact, from \Cref{thm:gin_gen} and \Cref{prop:0/1_vertices}, $\SP(\gin(J_G))=\SP(I_G)$ if and only if $G$ is complete. It is also reasonable to ask whether their projections to the first $n$ coordinates are equal, that is, $\pi(\SP(\gin(J_G)))=\pi(\SP(I_G))$ under the map $\pi:\mathbb{R}^{2n}\rightarrow \mathbb{R}^n$ defined by $(\pi(a))_i=a_i$. As the following example shows, this is not the case.

\begin{example}
	The graph below gives the following vertex sets for $\SP(\gin(J_G))$ and $\SP(I_G)$, expressed here as vectors. In line with \Cref{thm:wald_equality}, $\wald(\gin(J_G))=\wald(I_G)=\frac{3}{2}$.
	
	\begin{center}
		\begin{minipage}{0.15\textwidth}
			\centering
			\begin{tikzpicture}[scale=1, every node/.style={circle, draw, inner sep=2pt}]
				\node (1) at (90:1) {1};
				\node (2) at (210:1) {2};
				\node (3) at (330:1) {3};
				\node[draw=none, inner sep=0pt] at (0,-1.2) {$G$};
				\node (4) at (90:2.4) {4};
				\draw (1) -- (2) -- (3) -- (1);
				\draw (1) -- (4);
			\end{tikzpicture}
		\end{minipage}\hfill
		\begin{minipage}{0.8\textwidth}
			\begin{flalign*}
				&\mathcal{V}(\SP(I_G))=\left\{\!
				\begin{bmatrix}1\\1\\0\\0\\0\\0\\0\\0\end{bmatrix},
				\begin{bmatrix}1\\0\\1\\0\\0\\0\\0\\0\end{bmatrix},
				\begin{bmatrix}0\\1\\1\\0\\0\\0\\0\\0\end{bmatrix},
				\tfrac{1}{2}\begin{bmatrix}1\\1\\1\\0\\0\\0\\0\\0\end{bmatrix},
				\begin{bmatrix}1\\0\\0\\1\\0\\0\\0\\0\end{bmatrix}
				\right\} \\ &\mathcal{V}(\SP(\gin(J_G)))=\mathcal{V}(\SP(I_G))\cup\left\{\!
				\begin{bmatrix}0\\1\\0\\1\\1\\0\\0\\0\end{bmatrix},
				\begin{bmatrix}0\\0\\1\\1\\1\\0\\0\\0\end{bmatrix},
				\tfrac{1}{2}\begin{bmatrix}1\\1\\0\\1\\1\\0\\0\\0\end{bmatrix},
				\tfrac{1}{2}\begin{bmatrix}0\\1\\1\\1\\1\\0\\0\\0\end{bmatrix},
				\tfrac{1}{3}\begin{bmatrix}1\\1\\1\\1\\1\\0\\0\\0\end{bmatrix}
				\right\}
			\end{flalign*}
		\end{minipage}
	\end{center}
	The projection of the last vertex under $\pi$ is $\tfrac{1}{3}\begin{bmatrix}
		1 \\ 1 \\ 1 \\ 1
	\end{bmatrix}$, while 
	$c\begin{bmatrix}
		1 \\ 1 \\ 1 \\ 1
	\end{bmatrix}\in \pi(\SP(I_G))$ only for $c\geq \frac{1}{2}$.
\end{example}

In \cite{SP}, $\wald(I_G)$ is described in terms of $G$'s fractional chromatic number. This is used to yield the following bound on $\wald(I_G)$. 
\begin{thm}\cite[Theorem 6.3]{SP}
	\label{thm:edge_ideal_chi_bound}
	Let $G$ be a non-empty graph with chromatic number $\chi(G)$ and clique number $\omega(G)$. Then if $I_G$ is the edge ideal of $G$, 
	$$\frac{\chi(G)}{\chi(G)-1}\leq \wald(I_G) \leq\frac{\omega(G)}{\omega(G)-1}.$$
\end{thm}
We give an alternate proof of this result using \Cref{thm:wald_equality}. To start, we consider the full vertices of complete graphs.
\begin{prop}
	\label{prop:complete_full_vertices}
	For $n\geq 2$,
	$$\mathcal{V}_F(\SP(\gin(J_{K_n})))=\left\{\frac{1}{n-1}(\underset{n}{\underbrace{1,\dots,1}},\underset{n}{\underbrace{0,\dots,0}})\right\}.$$
\end{prop}
\begin{proof}
	The only IDS of $K_n$ is $T=\varnothing$, corresponding to $n$ one-facets of $\SP(\gin(J_{K_n}))$ with the following inequality for a given $1\leq j\leq n$: $$\left(\sum_{i=1}^nx_i\right)-x_j\geq 1.$$ Since no $y$'s appear in the one-facets, any $a\in \mathcal{V}_F(\SP(\gin(J_{K_n})))$ has $a_{n+1}=\dots=a_{2n}=0$. This implies $a_1,\dots,a_n>0$ and $a$ lies on exactly $n$ zero-facets, so by \Cref{lem:vertex_criterion}, $a$ must lie on all $n$ one-facets. That is, for any $1\leq c<d\leq n$, 
	$$\left(\sum_{i=1}^na_i\right)-a_c=\left(\sum_{i=1}^na_i\right)-a_d=1.$$ Canceling gives $a_c=a_d$, so $a_1=\dots=a_n$ and the claim follows after verifying $a=\frac{1}{n-1}(1,\dots,1,0,\dots,0)$ is indeed a vertex. 
	
	Take ${\bf u}\in \mathbb{R}^{2n}$ and assume $a\pm \epsilon {\bf u}\in \SP(\gin(J_{K_n}))$ for some $\epsilon >0$. For a given $1\leq i\leq n$, to satisfy the zero-facet corresponding to the $(n+i)$-th coordinate, we have $(a\pm\epsilon{\bf u})_{n+i}=\pm \epsilon {\bf u}_{n+i}\geq 0$, which implies ${\bf u}_{n+i}=0$. So ${\bf u}_{n+1}=\dots={\bf u}_{2n}=0.$ Per \Cref{obs:facets}, since $a$ is on every one-facet, so is $a\pm \epsilon {\bf u}.$ Then for any $1\leq c<d\leq n$,
	$$\left(\sum_{i=1}^n(a\pm \epsilon {\bf u})_i\right)-(a\pm \epsilon {\bf u})_c=\left(\sum_{i=1}^n(a\pm \epsilon {\bf u})_i\right)-(a\pm \epsilon {\bf u})_d=1.$$ 
	Canceling gives $(a\pm \epsilon {\bf u})_c=(a\pm \epsilon {\bf u})_d$, and since $a_c=a_d$, we have ${\bf u}_c={\bf u}_d$. Then ${\bf u}_1=\dots = {\bf u}_n$.
	
	But $\left(\sum_{i=1}^na_i\right)-a_c=1$, so $$\left(\sum_{i=1}^n(a\pm \epsilon {\bf u})_i\right)-(a\pm \epsilon {\bf u})_c=1\pm \epsilon \left(\left(\sum_{i=1}^n {\bf u}_i\right)-{\bf u}_c\right)=1,$$ and we conclude ${\bf u}_1=\dots = {\bf u}_n=0.$ Therefore, ${\bf u}=0$, and $a$ is a vertex per \Cref{defn:facets}.
\end{proof}
This proposition lets us easily compute $\wald(\gin(J_{K_n}))$.
\begin{cor}
	\label{cor:complete_wald}
	$$\wald(\gin(J_{K_n}))=\frac{n}{n-1}.$$
\end{cor}
\begin{proof}
	 Define $\mathcal{H}_G$ as in \Cref{cor:subgraph_decomp}. Then for any $H\in \mathcal{H}_G$, $H\cong K_m$ for some $2\leq m\leq n$. For $a\in\mathcal{V}_F(\SP(\gin(J_H)))$,  $\sum_{i=1}^{2n}a_i=\frac{m}{m-1}$ by \Cref{prop:complete_full_vertices}. By  \Cref{cor:subgraph_decomp}, $\wald(\gin(J_{K_n}))=\min\left\{\frac{m}{m-1}\mid 2\leq m\leq n\right\}=\frac{n}{n-1}$.
\end{proof}
We can use this result, along with \Cref{thm:vertex_containment}, to generalize to complete $k$-partite graphs.
\begin{prop}
	\label{prop:complete_kpartite}
	Let $G=K_{c_1,\dots,c_k}$ be a complete $k$-partite graph, with $c_1+\dots+c_k=n$. Then $$\widehat{\alpha}(\gin(J_G))=\frac{k}{k-1}.$$
\end{prop}
\begin{proof}
	Let $C_1,\dots, C_k$ be the maximal sets of independent vertices forming the $k$-parition of $G$, with $|C_i|=c_i$ for $1\leq i\leq k$.
	We label the vertices $C_i=\{1^{(C_i)},\dots,c_i^{(C_i)}\}$ to distinguish the sets. Assuming that $$c_1=c_2=\dots=c_\alpha=1<c_{\alpha+1}\leq \dots \leq c_k,
	$$
	the IDS's of $G$ are $T_{\alpha+1},\dots,T_k,$ and $\varnothing$, where for $\alpha+1\leq m\leq k$, $T_m=[n]\setminus C_m.$ Since $G\setminus T_{m}=C_m$, where $C_m$ is viewed as the edgeless graph on $c_m$ vertices, each $T_m$ is associated to only one minimal prime of $\gin(J_G)$: $P_{T_m,C_m}$, in the notation of \Cref{prop:new_gin_prime_decomp}.
	These primes give $k-\alpha$ one-facets $F_m$ with inequality
	$$\sum_{i=1}^{k}\sum_{j=1}^{c_i}(x_{j^{(C_i)}}+y_{j^{(C_i)}})
	-\sum_{j=1}^{c_m} (x_{j^{(C_m)}}+y_{j^{(C_m)}})\geq 1.$$ There are only $n$ other minimal primes: $P_{\varnothing, \{\ell\}}$, for each $1\leq \ell\leq n$. These give one-facets with the following inequality:
	$$\sum_{i=1}^k\sum_{j=1}^{c_i}x_{j^{(C_i)}}-x_\ell\geq 1.$$
	Take $a\in \mathcal{V}(\SP(\gin(J_G))),$ and define $b_i\coloneqq \sum_{j=1}^{c_i} (a_{j^{(C_i)}}+a_{n+j^{(C_i)}})$. Then for $\alpha+1\leq m\leq k$ one has
	\begin{equation}
		\label{facets1}
		\sum_{i=1}^k b_i-b_m=\sum_{i=1}^{k}\sum_{j=1}^{c_i}(a_{j^{(C_i)}}+a_{n+j^{(C_i)}})
		-\sum_{j=1}^{c_m} (a_{j^{(C_m)}}+a_{n+j^{(C_m)}})\geq 1.
	\end{equation} If $m\leq \alpha$, then
	\begin{equation}
		\label{facets2}
		\sum_{i=1}^k b_i-b_m=\sum_{i=1}^{k}\sum_{j=1}^{c_i}(a_{j^{(C_i)}}+a_{n+j^{(C_i)}})
		- a_{1^{(C_m)}}-a_{n+1^{(C_m)}}\geq \sum_{i=1}^k\sum_{j=1}^{c_i}a_{j^{(C_i)}}-a_{1^{(C_m)}} \geq 1.
	\end{equation}
	Considering the map $\pi:\mathbb{R}^{2n}\rightarrow \mathbb{R}^{2k}$ defined by $(\pi(a))_i=\begin{cases}
		b_i, & i\leq k\\
		0, &i\geq k+1
	\end{cases}$, we see by \eqref{facets1} and \eqref{facets2} that $\pi(\SP(\gin({J_G})))\subset \SP(\gin(J_{K_k}))$.  So by \Cref{cor:subgraph_decomp} and \Cref{prop:complete_full_vertices},  the following inequality holds
	$$\sum_{j=1}^{2n} a_j\geq \min\left\{1+\max\{b_1,\dots,b_k\}\mid b\in \SP(\gin(J_{K_k}))\right\}=1+\frac{1}{k-1}=\frac{k}{k-1}.$$
	Taking the induced subgraph $H$ with $V_H=\{1^{(C_1)},\dots,1^{(C_k)}\}$, $H\cong K_k$. Then by 
	\Cref{cor:subgraph_decomp} and \Cref{cor:complete_wald}, we obtain the inequality
	$$\wald(\gin(J_G))\leq \wald(\gin(J_H))=\frac{k}{k-1}$$
	and the proposition follows.
\end{proof}

The following lemma is the final piece we need, relating a graph, a vertex coloring, and its corresponding multipartite supergraph.
\begin{lem}
	\label{lem:partitions_to_complete_partite}
	Partition a connected graph $H$ into independent sets of vertices $C_1,\dots, C_k$, with $|C_i|=c_i$ and $C_i=\{1^{(C_1)},\dots,c_i^{(C_i)}\}$. Let $G=K_{c_1,\dots,c_k}$ be the complete $k$-partite graph, with $C_1$ through $C_k$ as the corresponding maximal sets of independent vertices. Then $$\SP(\gin(J_H))\subseteq \SP(\gin(J_G)).$$ 
\end{lem}
\begin{proof}
	Again assume $c_1=\dots=c_\alpha=1<c_{\alpha+1}\leq \dots\leq c_k$. Then as in \Cref{prop:complete_kpartite}, the IDS's of $G$ are $\widetilde{T_{\alpha+1}},\dots,\widetilde{T_k},$ 
	and $\varnothing$,  where for $\alpha+1\leq m\leq k$, $\widetilde{T_m}=[n]\setminus C_m.$ Each $\widetilde{T}_m$ corresponds to a one-facet $F_m$, and $T=\varnothing$ has $n$ corresponding one-facets. The minimal primes $P_{\varnothing, \{\ell\}}$ encoded by the latter $n$ one-facets are common to $\gin(J_H)$ and $\gin(J_G)$, since $V_H=V_G$ and $H$ is connected. We show that any $a\in \SP(\gin(J_H))$ satisfies the other $k-\alpha$ one-facet inequalities. 
	
	Fix $m$ satisfying $\alpha+1\leq m\leq k$ and let $T_{m,0}=\widetilde{T_m}$. For $i\geq 0$, if there is $v\in T_{m,i}$ such that $c_H(T\setminus\{v\})\geq c_H(T)$, let $T_{m,i+1}=T_{m,i}\setminus\{v\}$. This algorithm will terminate since $H$ is finite and connected, and $c_H(T_0)=c_G(\widetilde{T_m})=c_m$; call its output $T_{m,\Omega}$. By construction, $T_{m,\Omega}$ is an IDS for $H$, and $T_{m,\Omega}\subset \widetilde{T_m}$. 
	
	Let $H_1,\dots,H_{c_H(T_{m,\Omega})}$ be the connected components of $H\setminus T_{m,\Omega}$. Each $H_j$ contains at most one vertex of $C_m$, since otherwise the number of connected components would have decreased at some point in the algorithm. If $H_j\cap C_m\neq \varnothing$, let $v_j$ be the unique element in $H_j\cap C_m$. Otherwise, pick $v_j\in H_j$ arbitrarily. Setting $V=\{v_1,\dots,v_{c_H(T_{m,\Omega})}\}$, we have the following minimal prime of $H$ from \Cref{prop:new_gin_prime_decomp}:
	$$P_{T_{m,\Omega},V}=(x_i,y_i\mid i\in T_{m,\Omega})+\sum_{j=1}^{c_H(T_{m,\Omega})}(x_i\mid i\in H_j \setminus \{v_j\}).$$
	Since $a$ satisfies the facet inequality
	$$\sum_{i\in T_{m,\Omega}}(a_i+a_{n+i})+\sum_{j=1}^{c_H(T_{m,\Omega})}\sum_{i\in H_j\setminus\{v_j\}}a_i\geq 1,$$ 
	and
	$$T_{m,\Omega}\cup \bigcup_{j=1}^{c_H(T_{m,\Omega})}H_j\setminus\{v_j\}\subset \widetilde{T}_m,$$ 
	we can conclude
	$$\sum_{i\in \widetilde{T_m}}(a_i+a_{n+i})\geq \sum_{i\in T_{m,\Omega}}(a_i+a_{n+i})+\sum_{j=1}^{c_H(T_{m,\Omega})}\sum_{i\in H_j\setminus\{v_j\}}a_i\geq 1.$$ Thus, $a$ satisfies all facet inequalities of $\SP(\gin(J_G))$, and the claim follows.
\end{proof}
We are ready to bound $\wald(\gin(J_G))$.
\begin{thm}
	\label{thm:chi_omega_bound}
	Let $G$ be a non-empty graph with chromatic number $\chi(G)$ and clique number $\omega(G)$. Then
	$$\frac{\chi(G)}{\chi(G)-1}\leq \widehat\alpha(\gin(J_G))=\wald(I_G)\leq \frac{\omega(G)}{\omega(G)-1}.$$
\end{thm}
\begin{proof}
	The equality is given by \Cref{thm:wald_equality}. For the upper bound, by definition, there is an induced subgraph $H\cong K_{\omega(G)}$ of $G$. Thus, by \Cref{cor:subgraph_decomp}, $\wald(\gin(J_G))\leq \wald(\gin(J_H))=\frac{\omega(G)}{\omega(G)-1}$.
	
	Finally, by definition, $G$ can be partitioned into non-empty independent sets $C_1,\dots, C_{\chi(G)}$. Then by \Cref{lem:partitions_to_complete_partite} and \Cref{prop:complete_kpartite},
	\begin{flalign*}
		\widehat\alpha(\gin(J_G))=&\min\left\{\sum_{i=1}^{2n} a_i \; \middle| \;a\in \SP(\gin(J_G))\right\}\\\geq &\min\left\{\sum_{i=1}^{2n} a_i \; \middle| \;a\in \SP(\gin(J_{K_{c_1,\dots,c_{\chi(G)}}}))\right\}=\frac{\chi(G)}{\chi(G)-1}.
	\end{flalign*}   
\end{proof}

Parts (i) and (iii) of \cite[Theorem 6.7]{SP} then also hold for $\gin(J_G)$, as demonstrated in the following  corollaries:
\begin{cor}
	Let $G$ be a weakly-perfect graph, meaning $\chi(G)=\omega(G)$. Then $$\wald(\gin(J_G))=\frac{\chi(G)}{\chi(G)-1}.$$
\end{cor}
\begin{cor}
	Let $G$ be a bipartite graph. Then $\widehat\alpha(\gin(J_G))=2$.
\end{cor}
\begin{proof}
	By assumption, $\chi(G)=2$. Since $\widehat\alpha(\gin(J_G))\leq \alpha(\gin(J_G))=2$ for any graph, by \Cref{thm:chi_omega_bound}, the result holds.
\end{proof}

\section{Open Problems}
We conclude with a few relevant open questions. In light of \Cref{closed_graph_cor}, a natural question arises of whether there exist graphs whose binomial edge ideals have an asymptotic regularity not equal to two. We believe that the net graph is such an example. 
\begin{conj}
	Letting $N$ be the net graph, shown in \Cref{fig:net},  $\areg(J_N)=3$.
\end{conj} 
Moreover, \Cref{cor:induced_path_lower_bound} gives a lower bound on the asymptotic regularity of $\areg(\gin(J_G))$. We can analogously bound $\areg(\inid(J_G))$ below by the length of the longest admissible path on $G$, which must be isomorphic to a path graph by definition. We are not aware of any graphs $G$ for which equality is not achieved. This motivates the following conjecture.
\begin{conj}
	For an enumerated graph $G$, let $\ell$ be the length of the longest induced path in $G$, and let $\ell_<$ be the length of the longest admissible path in $G$.  
	Then
	\begin{flalign*}
		\areg(\gin(J_G))&=\ell, \\
		\areg(\inid(J_G))&=\ell_<.
	\end{flalign*}
\end{conj}
\smallskip
{\bf Acknowledgements.}
We would like to thank all those involved in the International REU in Guanajuato, especially the organizers Eloísa Grifo, Jack Jeffries,
Luis Núñez Betancourt, Alexandra Seceleanu, Adam LaClair, Pedro Angel Ramírez Moreno, and Shahriyar Roshan Zamir. We specifically like to thank Alexandra Seceleanu and Adam LaClair for many helpful comments and discussions.

This research was funded by NSF RTG Grant DMS-2342256 \emph{Commutative Algebra at Nebraska} and by SECIHTI grants CF-2023-G-33 \emph{Redefiniendo fronteras entre el \'algebra conmutativa, la teor\'ia de c\'odigos y la teor\'ia de singularidades} and Grant CBF 2023-2024-224: \emph{\'Algebra conmutativa, singularidades y c\'odigos}, as part of the the International REU in Commutative Algebra, held at CIMAT in Summer 2025, and organized by CIMAT and the University of Nebraska-Lincoln.
	
	\bibliographystyle{amsplain}
	\bibliography{references}
	
\end{document}